\RequirePackage{etex}
\RequirePackage{easymat}

\documentclass[12pt]{elsarticle}

\usepackage{graphicx,amsmath,amsthm, amssymb,MnSymbol,extsizes,fge,mathrsfs}

\usepackage[all]{xy}

\usepackage[T1]{fontenc}
\usepackage[active]{srcltx}
\renewcommand{\le}{\leqslant}
\renewcommand{\ge}{\geqslant}

\newcommand{\re}{\ \rhookdownarrow\ }

\newcommand{\matt}[1]{\left[\begin{smallmatrix}
 #1 \end{smallmatrix}\right]}
\newcommand{\matp}[1]{\!\begin{smallmatrix}
 #1\end{smallmatrix}\!}
\newcommand{\mat}[1]{\begin{bmatrix}
 #1  \end{bmatrix}}
\newcommand{\ma}[1]{\begin{matrix}
 #1  \end{matrix}}
\newcommand{\arr}[1]{\left[\begin{array}
 #1\end{array}\right]}

\newcommand{\arrr}[1]%
{\left[\begin{array}{ccc|cccc}
#1\end{array}\right]}

\newcommand{\mc}{\mathcal}
\newcommand{\mb}{\mathbb}
\newcommand{\wt}{\widetilde}

\newcommand{\D}{\Delta}

\newcommand{\da}{0^{\downarrow}}
\newcommand{\ua}{0^{\uparrow}}
\newcommand{\ra}{0^{\rightarrow}}
\newcommand{\la}{0^{\leftarrow}}

\renewcommand{\ll}{\langle}
\newcommand{\rr}{\rangle}

\newcommand{\un}{\underline}
\newcommand{\up}{\overline}

\DeclareMathOperator{\nullity}{nullity}

\newtheorem{theorem}{Theorem}[section]

\newtheorem*{fmtheorem}{Theorem I}
\newtheorem*{smtheorem}{Theorem II}
\newtheorem*{tmtheorem}{Theorem III}
\newtheorem*{fotheorem}{Theorem IV}

\newtheorem{lemma}{Lemma}[section]

\theoremstyle{definition}

\theoremstyle{remark}
\newtheorem{remark}{Remark}[section]
\newtheorem{example}{Example}[section]

\newcommand{\ze}{\fgestruckzero}

\newcommand{\ddotc}{
\text{\begin{picture}(9,6)
    \put(6,-3){$\cdot$}
\put(3,0){$\cdot$}
\put(0,3){$\cdot$}
\end{picture}}}

\newcommand{\ddotl}{
\text{\begin{picture}(0,0)
    \put(7,-3){$\cdot$}
\put(0,0){$\cdot$}
\put(-7,3){$\cdot$}
\end{picture}}}

\newcommand{\vdotc}{
\text{\begin{picture}(6,6)
    \put(1,-3){$\cdot$}
\put(1,0){$\cdot$}
\put(1,3){$\cdot$}
\end{picture}}}

\begin{document}

\title{A constructive proof of Pokrzywa's    theorem about perturbations of matrix pencils}

\author[fut]{Vyacheslav Futorny}
\ead{futorny@ime.usp.br}
\address[fut]{Department of Mathematics, University of S\~ao Paulo, Brazil}

\author[kly]{Tetiana Klymchuk}
\ead{tetiana.klymchuk@upc.edu}
\address[kly]{Universitat Polit\`{e}cnica de Catalunya, Barcelona, Spain}

\author[ser]{Vladimir V.
Sergeichuk\corref{cor}}
\ead{sergeich@imath.kiev.ua}
\address[ser]{Institute of Mathematics,
Kyiv, Ukraine}
\cortext[cor]{Corresponding author}

\author[shv]{Nadya Shvai}
\ead{nadiia.shvai@gmail.com}
\address[shv]{National University ``Kyiv-Mohyla Academy'',
Kyiv, Ukraine}

\begin{abstract}
Our purpose is to give new proofs of several known results about perturbations of matrix pencils.
Andrzej Pokrzywa (1986) described the closure of orbit of a Kronecker canonical pencil $A-\lambda B$ in terms of inequalities with pencil invariants. In more detail, Pokrzywa described all Kronecker canonical pencils $K-\lambda L$ such that each neighborhood of $A-\lambda B$ contains a pencil whose Kronecker canonical form is $K-\lambda L$. Another solution of this problem was given by Klaus Bongartz (1996) by methods of representation theory.

We give a direct and constructive proof of Pokrzywa's theorem.
We reduce its proof to the cases of matrices under similarity and of matrix pencils $P-\lambda Q$ that are direct sums of two indecomposable Kronecker canonical pencils. We
calculate the Kronecker forms of all pencils in a neighborhood of such a pencil $P-\lambda Q$.
In fact, we calculate the Kronecker forms of only those pencils that belong to  a miniversal deformation of $P-\lambda Q$, which is sufficient since all pencils  in a neighborhood of $P-\lambda Q$ are reduced to them by smooth strict equivalence
transformations.
\end{abstract}

\maketitle

\section{Introduction}

For each Jordan matrix $A$, all Jordan matrices $J$ such that each neighborhood of $A$ contains a matrix that is similar to $J$ have been described by
Den  Boer and Thijsse \cite{den} and, independently, by Markus and Parilis \cite{mar}. Pokrzywa \cite{pok} extends their results to Kronecker canonical pencils $A-\lambda B$ ($A,B\in\mathbb C^{m\times n}$): he describes the set of all Kronecker canonical pencils $K-\lambda L$ such that each neighborhood of $A-\lambda B$ contains a pencil whose Kronecker canonical form is $K-\lambda L$. Pokrzywa formulates and proves his theorem in terms of inequalities for invariants of matrix pencils. A more abstract solution of this problem is given by Bongartz \cite[Section 5, Table I]{bon} by methods of representation theory (see also \cite{ben,bon_1,bon_2,bon_3}).

The main purpose of our paper is to give a direct and constructive proof of Pokrzywa's theorem.
Instead of pencils $A-\lambda B$, we consider  matrix  pairs $(A,B)$ in which both matrices have the same size. We study them up to \emph{equivalence transformations} $(SAR,SBR)$, in which $S$ and $R$ are nonsingular. The \emph{orbit} $\ll\mc A\rr$ of $\mc A=(A,B)$ is
the set of all pairs that are equivalent to $\mc A$. Let $P_{m,n}$ be the partially ordered set, whose elements are the orbits of pairs of $m\times n$ matrices with the following ordering: $\ll\mc A\rr\le\ll\mc B\rr$ if $\ll\mc A\rr$ is contained in the closure of $\ll\mc B\rr$. Thus,
\begin{equation}
\label{bbb}
\parbox[c]{0.84\textwidth}{$\ll\mc A\rr\le \ll\mc B\rr$ if and only if
a pair that is equivalent to
$\mc B$ can be obtained by an arbitrarily small perturbation of   $\mc A$.}
\end{equation}

An orbit
$\ll\mc B\rr$ \emph{immediately succeeds} $\ll\mc A\rr$ (many authors write that $\ll\mc B\rr$ \emph{covers} $\ll\mc A\rr$; see \cite{kag2}) if $\ll\mc A\rr< \ll\mc B\rr$ and there exists no $\ll\mc C\rr$ such that
$\ll\mc A\rr<\ll\mc C\rr< \ll\mc B\rr$.

The \emph{Hasse diagram} of $\mc P_{m,n}$ (which is also called the \emph{closure graph} of the orbits of pairs of $m\times n$ matrices) is the directed graph whose vertices are all elements of $\mc P_{m,n}$ and there is an arrow $\ll\mc A\rr\to \ll\mc B\rr$ if and only if  $\ll\mc B\rr$ immediately succeeds $\ll\mc A\rr$.

For example, each pair of $1\times 2$ matrices is equivalent to exactly one of the pairs
\[
([0\ 0],[0\ 0]),\quad ([1\ 0],[\lambda \ 0]),\quad
([0\ 0],[1\ 0]),\quad  ([1\ 0],[0\ 1]),
\]
in which $\lambda \in \mathbb C$ (see \eqref{iub}). The Hasse diagram of $\mc P_{1,2}$ is
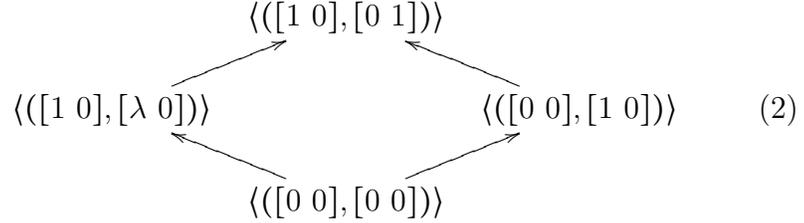
\begin{equation}\label{g1}
\begin{split}\qquad\qquad
\xymatrix@R=17pt@C=8pt{
&\ll([1\ 0],[0\ 1])\rr&
                    \\
\ll([1\ 0],[\lambda \ 0])\rr\ar[ur]
 && \ll([0\ 0],[1\ 0])\rr\ar[ul]
                    \\
&\ll([0\ 0],[0\ 0])\rr\ar[ul]\ar[ur]&
} \end{split}
\end{equation}
By \eqref{bbb}, for each arrow $\ll\mc A\rr\to \ll\mc B\rr$ there exists an arbitrarily small perturbation $\Delta \mc A$ such that $\mc A+\Delta \mc A$ is equivalent to $\mc B$; we locate $\Delta \mc A$  on the corresponding arrow in \eqref{g1}:
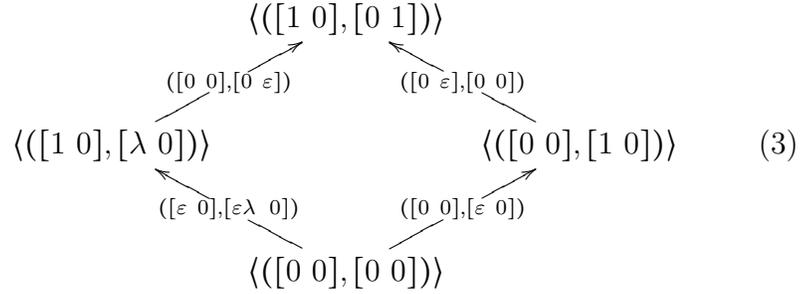
\begin{equation}\label{g3}
\begin{split}\qquad\qquad
\xymatrix@R=30pt@C=8pt{
&\ll([1\ 0],[0\ 1])\rr&
                    \\
\ll([1\ 0],[\lambda \ 0])\rr\ar[ur]|{([0\ 0],[0\ \varepsilon ])}
 && \ll([0\ 0],[1\ 0])\rr\ar[ul]|{([0\ \varepsilon ],[0\ 0])}
                    \\
&\ll([0\ 0],[0\ 0])\rr\ar[ul]|{([\varepsilon \ 0 ],[\varepsilon \lambda\ \, 0])}\ar[ur]|{([0\ 0],[\varepsilon  \ 0])}&
} \end{split}
\end{equation}
in which $\varepsilon$ is an arbitrarily small complex number.

The Hasse diagram of $\mc P_{2,3}$ is given in \cite{f-k}. The software StratiGraph \cite{f-j-k,K-J,str} constructs the Hasse diagram of  $\mc P_{m,n}$ for arbitrary $m$ and $n$. The Hasse diagrams for congruence classes of $2\times 2$ and $3\times 3$ complex matrices and for *congruence classes of $2\times 2$ complex matrices are constructed in \cite{d-f-k,f-k-s}. Orbit closures of matrix pencils are also studied in \cite{hin}.

The main theorems of the paper are formulated in Section \ref{sss3}. \emph{Theorem I} from Section \ref{sss3} is another form of Pokrzywa's theorem; it gives
replacements (i)--(vi) of direct summands such that a Kronecker pair $\mc A$ is transformed to a Kronecker pair $\mc B$ by a sequence of replacements of types (i)--(vi) if and only if $\ll\mc A\rr<\ll\mc B\rr$. Those replacements of $\mc A$ by $\mc B$
of types (i)--(vi) for which $\ll\mc B\rr$ is an immediate successor of $\ll\mc A\rr$ are given in \emph{Theorem II}, which is derived from Theorem I in Section \ref{sdsx}.

Two main tools in our proof of Theorem I are the following:
\begin{itemize}
  \item \emph{Theorem \ref{th}} from Section \ref{nfl}, which states that each immediate successor of the orbit of a Kronecker pair $\mc A$ is the orbit of a pair that is obtained by an arbitrarily small perturbation of only one subpair of $\mc A$ of two types: a direct sum of two indecomposable direct summands of $\mc A$, or a pair $(I,J)$, in which $J$ is a Jordan matrix with a single eigenvalue $\lambda \ne 0$. Thus, it is sufficient to prove Theorem I for such pairs $(I,J)$ and for direct summands of two indecomposable Kronecker pairs; we do this in Sections \ref{pairs} and \ref{jord}.

  \item The \emph{miniversal deformation} of a matrix pair under equivalence that is given by Garc\'{\i}a-Planas and Sergeichuk
\cite{gar_ser}; it is presented in Section \ref{prel}. In Section \ref{pairs}, we calculate the Kronecker forms of pairs that are obtained by arbitrary small perturbations of $(P,Q)$. In fact, we calculate the Kronecker forms of only those pairs of simple form that belong to  the miniversal deformation of $(P, Q)$, which is sufficient since all pencils  in a neighborhood of $(P,Q)$ are reduced to them by smooth equivalence
transformations.
\end{itemize}

\section{Main theorems}
\label{sss3}

All matrices that we consider are complex matrices and both matrices in each matrix pair have the same size.
For each positive
integer $n$, we define
the matrices
\begin{equation*}       
L_{n}:=\mat{
         1&0&&0\\
         &\ddots&\ddots\\
         0&&1&0},\quad
R_{n}:=\mat{
         0&1&&0\\
         &\ddots&\ddots\\
         0&&0&1}
\quad
\text{($(n-1)$-by-$n$)},
\end{equation*}
\[
J_n(\lambda) :=
\mat{\lambda&1&&0\\
  &\lambda&\ddots&\\
  &&\ddots&1\\
0&&&\lambda}\quad
\text{($n$-by-$n$,
}\lambda\in\mathbb C).
\]
(We denote by $0_{pq}$ the nonzero matrix of size $p\times q$ for all nonnegative integers $p$ and $q$. In particular, $L_1=R_1=0_{01}$. If $M$ is an $m\times n$ matrix, then $M\oplus 0_{0q}=\mat{M&0_{mq}}$ and
$M\oplus 0_{p0}=\matt{M\\0_{pn}}$.)

We also define the matrices
\[
0^{\nwarrow}:=\mat{1\,0\dots0\\0},\ \
0^{\nearrow}:=\mat{0\dots0\,1\\0},\ \
0^{\swarrow}:=\mat{0\\1\,0\dots0},\ \
0^{\searrow}:=\mat{0\\0\dots0\,1},
\]
whose sizes will be clear from the context.

The matrix pairs
\begin{equation}\label{iub}
\begin{gathered}
\mc L_n:=(L_n,R_n),\qquad \mc L_n^T:=(L_n^T,R_n^T),
\\
{\mc D}_n(\lambda ):=
  \begin{cases}
(I_n,J_n(\lambda))& \text{if $\lambda \in\mathbb C$} \\
(J_n(0),I_n) & \text{if $\lambda=\infty$}
  \end{cases}
\end{gathered}
\end{equation}
are called \emph{indecomposable Kronecker pairs}. Leopold Kronecker proved that each matrix pair $\mc A$ is equivalent to a direct sum of indecomposable Kronecker pairs. This direct sum  is called the \emph{Kronecker canonical form} of $\mc A$; it is determined by $\mc A$ uniquely, up to permutations of direct summands.

Pokrzywa describes the closures of orbits of Kronecker canonical pencils in Theorem 3 from \cite{pok}, which is
formulated and proved in the form of systems of inequalities for invariants of matrix pencils with respect to strict equivalence (see also \cite[Theorem 2.1]{de} and \cite[Theorem 3.1]{kag2}).  However, he formulates Lemma 5 from \cite{pok} in the form of replacements of direct summands of Kronecker pairs; such replacements are also given in \cite[Section 5.1]{bol} and \cite[Theorem 2.2]{dmy}. In the proof of Lemma 5 from \cite{pok}, Pokrzywa also gives arbitrarily small perturbations that ensure these replacements.
We describe the closures of orbits of Kronecker pairs in the following theorem.

\begin{fmtheorem}\label{thg}
Let $\mc A$ and $\mc B$ be nonequivalent Kronecker pairs. Then $\ll\mc A\rr<\ll \mc B\rr$ if and only if $\mc B$ can be obtained from $\mc A$ by permutations of direct summands and replacements  of direct summands of types {\rm(i)--(vi)} listed below, in which $m,n\in\{1,2,\dots\}$ and $\lambda \in\mb C\cup\infty$. The notation $\mc P\!\!\! \re\!\!\! \mc Q$ means that $\mc P$ is  replaced by $\mc Q$. For each replacement $\mc P\!\!\! \re\!\!\! \mc Q$, we also give a pair that is obtained by an arbitrarily small perturbation (which is defined by an arbitrary nonzero complex number $\varepsilon $) of $\mc P$ and whose Kronecker canonical form is $\mc Q$.
\begin{itemize}
\item[\rm(i)]
$\mc L_m^T \oplus
\mc L_n^T\re\mc L_{m+1}^T \oplus
\mc L_{n-1}^T$ in which $m+2\le n$, via  the pair
\[
\left(
\mat{L^T_m&0\\0&L^T_n},
\mat{R^T_m&\varepsilon 0^{\nwarrow}
\\0&R^T_n}
\right),
\]
which is obtained by a perturbation of $\mc L_m^T \oplus
\mc L_n^T$.

  \item[\rm(ii)]
$\mc L_m \oplus
\mc L_n \re
\mc L_{m+1} \oplus
\mc L_{n-1}$ in which $m+2\le n$, via
\[
\left(
\mat{L_m&0\\0&L_n},
\mat{R_m&0
\\\varepsilon 0^{\nwarrow}&R_n}
\right).
\]

\item[\rm(iii)]
$\mc L_m^T\oplus {\mc D}_n(\lambda ) \re
\mc L_{m+1}^T \oplus \mc D_{n-1}(\lambda )$ (the summands ${\mc D}_0(\lambda )$ are omitted), via
\[
\setlength{\arraycolsep}{2pt}
\left(
 \mat{
L^T_m&0\\ 0&I_n},\,
\mat{
R^T_m&\varepsilon 0^{\nearrow}
\\ 0&J_n(\lambda )}
 \right)
\,\text{if }\lambda \in\mathbb C,
           \
\left(
 \mat{
L^T_m&\varepsilon 0^{\searrow}\\ 0&J_n(0)},\,
\mat{
R^T_m&0\\ 0&I_n}
 \right)
\,\text{if }\lambda =\infty.
\]

\item[\rm(iv)]
$\mc L_m\oplus {\mc D}_n(\lambda ) \re
\mc L_{m+1} \oplus \mc D_{n-1}(\lambda )$, via
\[
\setlength{\arraycolsep}{2pt}
\left(
 \mat{
L_m&0\\ 0&I_n},\,
\mat{
R_m&0
\\ \varepsilon 0^{\nwarrow}&J_n(\lambda )}
 \right)\,
 \text{if }\lambda \in\mathbb C,
           \
\left(
 \mat{
L_m&0\\ \varepsilon 0^{\nearrow}&J_n(0)},\,
\mat{
R_m&0\\ 0&I_n}
 \right)\,
\text{if }\lambda =\infty.
\]

\item[\rm(v)]
$\mc D_m(\lambda)\oplus\mc D_n(\lambda)\re\mc D_{m-1}(\lambda)
\oplus \mc D_{n+1}(\lambda)$ in which $m\le n$, via
\[
\setlength{\arraycolsep}{2pt}
\left(I_{m+n},\mat{J_m(\lambda )&\varepsilon 0^{\nwarrow}\\0&J_n(\lambda )}\right)\,\text{if }\lambda \in\mathbb C,
                        \
\left(\mat{J_m(0)&\varepsilon 0^{\nwarrow}\\0&J_n(0)},I_{m+n}\right)
\,\text{if }\lambda =\infty.
\]

\item[\rm(vi)]
$\mc L_m^T\oplus
\mc L_n \re
{\mc D}_{r_1} (\mu_1)\oplus\dots\oplus {\mc D}_{r_k}(\mu_k)$, in which $\mu_1,\dots,\mu_k\in{\mathbb C}\cup\infty$ are distinct and $r_1+\dots +r_k=m+n-1$, via
\begin{equation}\label{kyb}
\arraycolsep=0.24em
\left(
 \left[\begin{array}{ccc|cccc}
1&&&&&&\alpha _1\\
0&\ddots&&&&&\alpha _2 \\[-3pt]
&\ddots&1&&
\text{\raisebox{7pt}[0pt][0pt]{$0$}}
&&\vdots\\
&&0&&&&\alpha _m\\\hline
&&&1&\,0\!\!\!\\
&0&&&\ddots&\ddots\\
&&&&&1\!\!&\,0\!\!
 \end{array}
              \right],\,
\left[\begin{array}{ccc|cccc}
0&&&\beta_1&\beta_2&\dots &\beta _n\\
1&\ddots&&&&&\\[-3pt]
&\ddots&0& &&
\!\!\!\!\!\!0&\\
&&1&&&&\\\hline
&&&0&1\\
&0&&&\ddots&\ddots\\
&&&&&0\!\!&1\!\!
 \end{array}
 \right]
 \right),
\end{equation}
in which
\begin{equation}\label{lyv}
(-\beta _1,\dots,-\beta _n,\alpha _1,\dots,\alpha _m):=
\varepsilon (c_0,\dots,c_{r-1},1,0,\dots,0),
\end{equation}
$\varepsilon$ is any nonzero complex number, and $c_{0},\dots,c_{r-1}$ are defined by
\begin{equation}\label{amt}
c_0+c_1x+\dots+c_{r-1}x^{r-1}+x^r
:=\prod_{\lambda_i\ne\infty}
(x-\lambda _i)^{r_i}.
\end{equation}
\end{itemize}
\end{fmtheorem}

The statements (i)--(vi) of Theorem I follow, respectively, from Theorems \ref{case1}--\ref{case2}
of Section \ref{pairs} due to Theorems \ref{th} and \ref{jut}.

Up to permutations of summands, each Kronecker pair has the form
\begin{equation}\label{hbg}
  \begin{gathered}
\mc A:=\bigoplus_{i=1}^{\un s}\mc L^T_{m_i}\oplus\bigoplus_{i=1}^{\up s}
\mc L_{n_i}\oplus
\bigoplus_{i=1}^{t}
\Big(\mc D_{k_{i1}}(\lambda_i)\oplus\dots\oplus
\mc D_{k_{is_i}}(\lambda_i)\Big),
                           \\
m_1\le\cdots\le m_{\underline s}, \ \
n_1\le\cdots\le n_{\overline s},\ \
k_{i1}\le\cdots\le k_{is_{i}}\ (i=1,\dots,t),
  \end{gathered}
\end{equation}
in which $\lambda_1,\dots,\lambda_t\in\mb C\cup\infty$ are distinct.
The numbers $\un s,s_1,\dots,s_t,\up s$ can be zero, which means that the corresponding direct summands in \eqref{hbg}
 are absent.


The following theorem in the form of coin moves is given by Edelman, Elmroth, and K\r{a}gstr\"{o}m \cite[Theorem 3.2]{kag2} (see also \cite[Theorem 2.4]{joh} and \cite{ben}).

\begin{smtheorem}
Let $\mc A$ be the Kronecker pair \eqref{hbg}.
An orbit $\mc O$ immediately succeeds $\ll\mc A\rr$ if and only if $\mc O$ is the orbit of a pair that is obtained from $\mc A$ by one of the following  replacements, which are special cases of the replacements {\rm(i)}--{\rm(vi)} from Theorem I:
\begin{itemize}

  \item[\rm(i$'$)]
$\mc L^T_{m_i} \oplus
\mc L^T_{m_j} \re
\mc L^T_{m_i+1} \oplus
\mc L^T_{m_j-1}$, in which either $j =i+1$ and $m_i+2\le m_{i+1}$, or $j =i+2$ and
$(m_i,m_{i+1},m_{i+2})=(m_i,m_i+1,m_i+2)$,

  \item[\rm(ii$'$)]
$\mc L_{n_i} \oplus
\mc L_{n_j} \re
\mc L_{n_i+1} \oplus
\mc L_{n_j-1}$, in which either $j =i+1$ and $n_i+2\le n_{i+1}$, or $j =i+2$ and $(n_i,n_{i+1},n_{i+2})=(n_i,n_i+1,n_i+2)$,

\item[\rm(iii$'$)]
$\mc L^T_{m_{\un s}}\oplus {\mc D}_{k_{is_i}}(\lambda_i) \re
\mc L^T_{m_{\un s}+1}\oplus {\mc D}_{k_{is_i}-1}(\lambda_i)$,

\item[\rm(iv$'$)]
$\mc L_{n_{\up s}}\oplus {\mc D}_{k_{is_i}}(\lambda_i) \re
\mc L_{n_{\up s}\,+1}\oplus {\mc D}_{k_{is_i}-1}(\lambda_i)$,

\item[\rm(v$'$)]
$\mc D_{k_{ij}}(\lambda_i)\oplus\mc D_{k_{i,j+1}}(\lambda_i)\re
\mc D_{k_{ij}-1}(\lambda_i)\oplus\mc D_{k_{i,j+1}+1}(\lambda_i)$,

\item[\rm(vi$'$)]
$\mc L^T_{m_{\un s}}\oplus
\mc L_{n_{\up s}}
 \re
{\mc D}_{r_1} (\mu_1)\oplus\dots\oplus {\mc D}_{r_q}(\mu_q)$, in which  $q\ge t$,
\begin{equation*}\label{klc}
\mu_1=\lambda _1\text{ and }k_{1s_1}\le r_1,\ \ \dots,\ \ \mu_t=\lambda _t\text{ and } k_{ts_t}\le r_t,
\end{equation*}
$\mu_1,\dots,\mu_q\in{\mathbb C}\cup\infty$ are distinct, and $r_1+\dots +r_q=m_{\un s}+n_{\up s}-1$.
\end{itemize}
\end{smtheorem}

Define the matrices whose sizes will be clear from the context:
\begin{equation}\label{rdpd}
\D_r(\varepsilon):=\mat{0...0\,\varepsilon  \,0...0\\0},\qquad
\nabla_r(\varepsilon)
:=\mat{0\\0...0\,\varepsilon  \,0...0},
\end{equation}
in which
$\varepsilon $ is an arbitrary nonzero complex number that is located in the $r$th column. We often write $\D_r$ and $\nabla_r$ omitting $\varepsilon$.  Set $\D_0=\nabla_0:=0$.

The \emph{lower cone} of an orbit $\ll\mc A\rr$ is the set $\ll\mc A\rr^{\vee}$ of all orbits $\ll\mc B\rr$ such that $\ll\mc A\rr\le\ll\mc B\rr$. Theorems \ref{case1}--\ref{case2} (which are used in the proof of Theorem I) ensure the following theorem.

\begin{tmtheorem}
If $\mc A$ is an indecomposable Kronecker pair, then
$\ll\mc A\rr^{\vee}$ is a one-element set; it consists of the orbit of  $\mc A$.

The lower cones of all direct sums of two indecomposable Kronecker pairs are the following ($\varepsilon $ is an arbitrary nonzero complex number).

\begin{itemize}

  \item[\rm(i)] The cone
$\ll\mc L_m^T \oplus \mc L_n^T\rr^{\vee}$ with $1\le m\le n$
consists of the orbits of
\begin{equation}\label{ffm}
\text{$\mc L_{m+r}^T \oplus
\mc L_{n-r}^T$,\quad in which $r\ge 0$ and $m+r\le n-r$.}
\end{equation}
Each pair \eqref{ffm} is the Kronecker canonical form of
\[\left(
\mat{L^T_m&0\\0&L^T_n},
\mat{R^T_m&\D_r(\varepsilon )
\\0&R^T_n}
\right).
\]

\item[\rm(ii)]
$\ll\mc L_m \oplus
\mc L_n\rr^{\vee}$ with $1\le m\le n$
consists of the orbits of
\begin{equation}\label{ffm1}
\text{$\mc L_{m+r} \oplus
\mc L_{n-r}$,\quad in which $r\ge 0$ and $m+r\le n-r$.}
\end{equation}
Each pair \eqref{ffm1} is the Kronecker canonical form of
\[\left(
\mat{L_m&0\\0&L_n},
\mat{R_m&0
\\\D_r^T(\varepsilon )&R_n}
\right).
\]

  \item[\rm(iii)]
$\ll\mc L^T_m\oplus {\mc D}_n(\lambda )\rr^{\vee}$ with $m\ge 1$, $n\ge 1$, and $\lambda \in\mb C\cup\infty$  consists of the orbits of
\begin{equation}\label{cjr}
\mc L^T_{m+r} \oplus \mc D_{n-r}(\lambda ),\quad\text{in which } 0\le r\le n.
\end{equation}
Each pair \eqref{cjr} with $r>0$ is the Kronecker canonical form of
\[
\setlength{\arraycolsep}{2pt}
    \left(
 \mat{
L^T_m&0\\ 0&I_n},\,
\mat{
R^T_m&\D_{n-r+1}
\\ 0&J_n(\lambda )}
 \right)\,
 \text{if }\lambda \in\mathbb C,
           \
\left(
 \mat{
L^T_m&\nabla_{n-r+1}\\ 0&J_n(0)},\,
\mat{
R^T_m&0\\ 0&I_n}
 \right)\,
\text{if }\lambda =\infty.
\]

\item[\rm(iv)]
$\ll\mc L_m\oplus{\mc D}_n(\lambda )\rr^{\vee}$ with $m\ge 1$, $n\ge 1$, and $\lambda \in\mb C\cup\infty$  consists of the orbits of
\begin{equation}\label{cgr}
\mc L_{m+r}\oplus \mc D_{n-r}(\lambda ) ,\quad\text{in which } 0\le r\le n.
\end{equation}
Each pair \eqref{cgr} with $r>0$ is the Kronecker canonical form of
\[
\setlength{\arraycolsep}{2pt}
    \left(
 \mat{
 L_m&0\\ 0&I_n},\,
\mat{
R_m&0
\\\D_{r}^T&J_n(\lambda )}
 \right)\,
 \text{if }\lambda \in\mathbb C,
           \
\left(
 \mat{
J_n(0)&\nabla_{r}^T\\0 &L_m },\,
\mat{
I_n&0\\ 0&R_m}
 \right)\,
\text{if }\lambda =\infty.
\]

  \item[\rm(v)]
$\ll\mc D_m(\lambda)\oplus\mc D_n(\lambda)\rr^{\vee}$ with $1\le m\le n$ and $\lambda \in\mb C\cup\infty$
consists of the orbits of
\begin{equation}\label{htm}
\text{$\mc D_{m-r}(\lambda)
\oplus \mc D_{n+r}(\lambda)$,\quad in which $0\le r\le m$.}
\end{equation}
Each pair \eqref{htm} with $r>0$ is the Kronecker canonical form of
\[
\setlength{\arraycolsep}{2pt}
\left(I_{m+n},\mat{J_m(\lambda )&\D_r^T\\0&J_n(\lambda )}\right)\,
\text{if }\lambda \in\mathbb C,
                        \
\left(\mat{J_m(0)&\D_r^T\\0&J_n(0)},I_{m+n}\right)\,
\text{if }\lambda =\infty.
\]

  \item[\rm(vi)]
$\ll\mc L^T_m\oplus
\mc L_n\rr^{\vee}$ with $m\ge 1$ and $n\ge 1$ consists of $\ll\mc L^T_m\oplus
\mc L_n\rr$ and the orbits of
\begin{equation}\label{akt}
{\mc D}_{r_1} (\mu_1)\oplus\dots\oplus {\mc D}_{r_k}(\mu_k),\qquad  \begin{matrix}
r_1+\dots +r_k=m+n-1, \\
\mu_1,\dots,\mu_k\in{\mathbb C}\cup\infty\text{ are distinct.} \\
\end{matrix}
\end{equation}
Each pair \eqref{akt} is the Kronecker canonical form of the pair \eqref{kyb} that is determined by \eqref{lyv} and
\eqref{amt}.
\end{itemize}
\end{tmtheorem}

Let us rearrange the direct summands of \eqref{hbg} as follows:
\begin{equation}\label{hwm}
  \begin{gathered}
  \begin{aligned}
\mc A:=&\mc L^T_{m_1}\oplus
\mc L^T_{m_2}\oplus\dots\oplus
\mc L^T_{m_{\underline s}}\\
&\oplus\bigoplus_{i=1}^{t}
\Big(\mc D_{k_{i1}}(\lambda_i)\oplus
\mc D_{k_{i2}}(\lambda_i)\oplus\dots\oplus
\mc D_{k_{is_i}}(\lambda_i)\Big)
                    \\
&\oplus\mc L_{n_{\up s}}\oplus
\mc L_{n_{\up s-1}}\oplus\dots\oplus
\mc L_{n_1},
  \end{aligned}
                           \\
m_1\le\cdots\le m_{\underline s}, \ \
k_{i1}\le\cdots\le k_{is_{i}}\ (i=1,\dots,t),\ \
n_1\le\cdots\le n_{\overline s}.
  \end{gathered}
\end{equation}
By the following theorem,
each immediate successor of $\ll\mc A\rr$ is the orbit of a pair that is obtained by an arbitrarily small perturbation of only one pair of  upper diagonal blocks of $\mc A$.

\begin{fotheorem}
Let $\mc A=([A_{ij}],[A'_{ij}])$ be the
Kronecker pair \eqref{hwm} partitioned into blocks $A_{ij}$ and $A'_{ij}$ such that the pairs of diagonal blocks $(A_{11}, A_{11}')$, $(A_{22},A_{22}'),$ $\ldots$ are the direct summands
\begin{equation}\label{weq}
\begin{split}
\mc L^T_{m_1},\dots,\mc
L^T_{m_{\un s}},\ \ &
\mc D_{k_{11}}(\lambda_1), \dots,D_{k_{1s_1}}(\lambda_1),\\ &\dots,\ \ \mc D_{k_{t1}}(\lambda_t), \dots,D_{k_{ts_t}}(\lambda_t),\ \
\mc L_{n_{\up s}},\dots,
\mc L_{n_1}
\end{split}
\end{equation}
of \eqref{hwm}. Then each immediate successor of $\ll\mc A\rr$ is the orbit of some matrix pair that is obtained from
$\mc A$ by an
arbitrarily small
perturbation of only one pair $(A_{ij},A'_{ij})$ with $i<j$
of its upper diagonal blocks.
\end{fotheorem}

Theorem IV follows from Theorem I due to the block-triangular forms of the perturbations that are given in (i)--(vi).
We move backwards in the next sections: we first give an independent proof of Theorem \ref{th}, which is a weak form of Theorem IV. Using it, we prove Theorem I in Sections \ref{pairs} and \ref{jord}.

\section{Preliminaries: miniversal deformations of matrix pencils}\label{prel}

The notion of a miniversal deformation of a square complex matrix $A$ under similarity was introduced by Vladimir Arnold in \cite{arn};  it is a family of matrices with the minimal number of parameters to which all matrices $B$ close to $A$ can be reduced by similarity
transformations that smoothly depend on the entries of $B$ (see formal definitions in \cite{arn,arn2,arn3}). For example, all matrices that are close to $J_3(\lambda )$ can be reduced to the form
\begin{equation}\label{yg4}
\mat{\lambda &1&0\\0&\lambda &1\\0&0&\lambda}+\mat{0&0&0\\0&0&0\\ *&*&*}
\end{equation}
by similarity transformations that are close to the identity.

Let us formulate Arnold's theorem.
We denote by $0^{\uparrow}_{pq}$ (respectively, $0^{\downarrow}_{pq}$, $0^{\leftarrow}_{pq}$, and $0^{\rightarrow}_{pq}$) the $p\times q$ matrix, in which all entries are zero except for the entries of the first row (respectively,  last row,  first column, and last column) that are stars.   We usually omit the indices $p$ and $q$. For example, the second matrix in \eqref{yg4} is $0^{\downarrow}_{33}$.

Let us arrange the Jordan blocks in a Jordan matrix with a single eigenvalue as follows:
\[
J_{k_1,\dots,k_s}(\lambda ):=
J_{k_1}(\lambda )\oplus\dots\oplus
J_{k_s}(\lambda ),\qquad
k_{1}\le k_{2}\le
\cdots\le k_{s},
\]
and define the matrix with stars:
\begin{equation}\label{hju}
\setlength{\arraycolsep}{-7pt}
\wt J_{k_1,\dots,k_s}(\lambda ):=
\begin{bmatrix}
  J_{k_{1}}(\lambda)+{\la}&
  {\la} &\qquad \dots\qquad & {\la}
  \\[7pt]
  {\da} & J_{k_{2}}(\lambda)+{\la}
   &\ddotl& \vdots
  \\[7pt]
  \vdots &\ddotl & \ddotl & {\la}
  \\[7pt]
  0^{\da} & \dots & {\da}
  & J_{k_{s}}(\lambda)+
  {\la}
\end{bmatrix}.
\end{equation}

The following theorem has been proved by Arnold \cite[Theorem 4.4]{arn}; see also \cite[Section 3.3]{arn2} and
\cite[\S\,30]{arn3}.

\begin{theorem}[{\cite{arn}}]\label{teo2}
Let a Jordan matrix be written in the form
\begin{equation}\label{kje}
J=J_{k_{11},\dots,k_{1s_1}}(\lambda _1)
\oplus\dots\oplus
J_{k_{l1},\dots,k_{ls_l}}(\lambda _l)
,
\quad
\begin{matrix}
k_{i1}\le k_{i2}\le
\cdots\le k_{is_i},
\\
\lambda _1,\dots,\lambda _l\in\mb C \text{ are distinct.}
\end{matrix}
\end{equation}
Then all matrices $J+X$
that are sufficiently close to $J$ can
be simultaneously reduced by some similarity
transformation
\begin{equation}\label{tef}
J+X\mapsto
S(X)^{-1}(J+X)
S(X),\qquad\begin{matrix}
\text{${S}(X)$
is analytic}\\
\text{at $0$ and }
S(0)=I,
\end{matrix}
\end{equation}
to the form
\begin{equation}\label{nyx}
\wt J:=\wt J_{k_{11},\dots,k_{1s_1}}(\lambda _1)
\oplus\dots\oplus
\wt J_{k_{l1},\dots,k_{ls_l}}(\lambda _l),
\end{equation}
in which
the stars are replaced by
complex numbers that depend
analytically on the entries of $X$. The
number of stars is minimal that can be
achieved by similarity transformations of the form
\eqref{tef}; this number is equal to the
codimension of the similarity class of
$J$.
\end{theorem}

Thus, the family of matrices \eqref{nyx} is a miniversal deformation of the Jordan matrix \eqref{kje}.

A constructive proof of Theorem \ref{teo2} by elementary transformations is given by Klimenko and Sergeichuk \cite{k-s}.
Many applications of miniversal deformations are given by Mailybaev \cite{mai1999,mai2000,mai2001}; he constructs a smooth similarity transformation \eqref{tef} in the form of Taylor series.
 The radius of a neighborhood of $J$ in which
all matrices $J+X$ are reduced
to the form \eqref{nyx} by transformations \eqref{tef} is calculated in \cite{bov}, in which Theorem \ref{teo2} is extended to matrices over the field of $p$-adic numbers.

A miniversal deformation of complex matrix pencils was constructed by
Edelman, Elmroth, and  K\r{a}gstr\"{o}m in the article \cite{kag}, which was awarded the SIAM Linear Algebra Prize 2000 for the most outstanding paper published in 1997--1999. However,
their miniversal deformations contain repeating parameters, which complicates their use in the proof of Theorem I. We use the simpler miniversal deformations constructed by Garc\'{\i}a-Planas and Sergeichuk \cite[Theorem 4.1]{gar_ser}.

Denote by $Z_{pq}$ the $p\times q$ matrix with $p\le q$, in which the first $\max\{q-p,0\}$ entries of the first row are the stars and the other entries are zeros:
\begin{equation*}       \label{3.1a}
Z_{pq}:=\mat{*&\dots&*&0&\dots&0\\
&&&&\ddots&\vdots\\&0&&&&0}
\end{equation*}
(we usually omit the indices $p$ and $q$).

\begin{theorem}[{\cite[Theorem 4.1]{gar_ser}}]      \label{first}
Let $\mc A$ be the Kronecker pair \eqref{hwm}, in which $\lambda _1,\dots,\lambda _{t-1}\in\mathbb C$ are distinct and $\lambda _t=\infty$.
Then all matrix pairs $\mc A+\mc X$
that are sufficiently close to $\mc A$ can
be simultaneously reduced by some equivalence
transformation
\begin{equation}\label{ehf}
\mc A+\mc X\mapsto
R(\mc X)^{-1}(\mc A+\mc X)
S(\mc X),\quad\begin{matrix}
\text{the matrices $R(\mc X)$ and $S(\mc X)$}\\
\text{are analytic at $(0,0)$,}\\
R(0,0)=I\text{ and } S(0,0)=I,
\end{matrix}
\end{equation}
to the form
\begin{equation}\label{aaa}
\left(
\left[
\begin{MAT}(@)[0.4pt]{cccc}
   \ma{L^T_{m_1}\hspace{-18pt}&&&0
   \\&L^T_{m_2}\hspace{-18pt}
   \\&&\ddotc\hspace{-10pt}\\&&&L^T_{m_{\underline s}}}
&0
&\,\ma{\da\\\da\\\vdots\\\da}
&\ma{\,\,\ra&\!\!...&\ra}
      \\
&\,I^{\rule{0pt}{6pt}}_{\rule{0pt}{6pt}}\,
&0&0
      \\
&&\widetilde J_0^{\rule{0pt}{10pt}}
&\ma{\,\,\ra&\!\!...&\ra}
      \\
0&&&\ma{L_{n_{\up s}}\hspace{-15pt}&&&0
\\&\ddotc\hspace{-10pt}
\\&&\!\!\! L_{n_2}\hspace{-10pt}
\\&&&\!\!\!L_{n_1}}
\addpath{(1,3,2)rrr}
\addpath{(3,4,2)ddd}
\addpath{(1,4,2)ddrrr}
\addpath{(2,4,2)dddrr}
\\
\end{MAT}
\right],
\left[
\begin{MAT}(@)[0.4pt]{cccc}
   \ma{R^T_{m_1}\!\!\!&\!\!\!Z\!\!\!&\!\!\!...\!\!\!&\!\!\!Z
   \\&\!\!\!R^T_{m_2}\!\!\!&\!\!\!\ddotc\!\!\!&\vdots
   \\&&\!\!\!\ddotc\!\!\!&\!\!\!Z
   \\&&&\!\!\!R^T_{m_{\underline s}}}
&\,\ma{\ua\\\ua\\\vdots\\\ua}
&0
&\ma{\ua\\\ua\\\vdots\\\ua}
      \\
&\widetilde J^{\rule{0pt}{10pt}}
&0&\ma{\ \la&\,...\,&\la}
      \\
&&\,I^{\rule{0pt}{6pt}}_{\rule{0pt}{6pt}}\,&0
      \\
0&&&   \ma{R_{n_{\up s}}\!\!\!&\!Z^T\!\!&\!\!\!\!...\!\!\!&Z^T\!
   \\&\!\!\!\ddotc\!\!\!&\!\!\ddotc\!\!&\!\!\vdots\!
   \\&&\!\!\!\!\!R_{n_2}\!\!\!\!&\!Z^T\!
   \\&&&\!\!\!R_{n_1}\!}
\addpath{(1,3,2)rrr}
\addpath{(3,4,2)ddd}
\addpath{(1,4,2)ddrrr}
\addpath{(2,4,2)dddrr}
\\
\end{MAT}
\right]
\right)
\end{equation}
in which
\[
\wt
J:=\bigoplus\limits_{i=1}^{t-1}
\wt J_{k_{i1},\dots,k_{is_i}}(\lambda _i)
,\qquad
\wt J_0:=\wt J_{k_{t1},\dots,k_{ts_t}}(0)
\]
(see \eqref{hju}) and the stars are replaced by complex numbers that depend
analytically on the entries of the pair $\mc X$. The
number of stars is minimal that can be
achieved by equivalence transformations of the form
\eqref{ehf}.
\end{theorem}

Note that the number of summands in \eqref{hwm} is $\ge 1$; i.e., the summands of each of the types and the corresponding horizontal and vertical strips in \eqref{aaa} can be absent.

By a \emph{miniversal pair} we mean a matrix pair that is obtained from \eqref{aaa} by replacing its stars by complex numbers. We use the Frobenius matrix norm
\begin{equation}\label{nor}
\|[a_{ij}]\|:=\sqrt{
\sum\nolimits_{ij}|a_{ij}|^2},\qquad a_{ij}\in\mathbb C.
\end{equation}
For  a matrix pair $\mc A=(A,A')$, we write $\|\mc A\|:=\|A\|+\|A'\|$ and define its neighborhood
\[
N_r(\mc A):=\{\mc B\,|\,\|\mc B-\mc A\|<r\},
\]
in which $r$ is a positive real number.

\begin{remark}\label{rrs}
Let $\mc A$ be the matrix pair from Theorem \ref{first}. Let $N_r(\mc A)$ be its neighborhood, in which all pairs are reduced to the form \eqref{aaa} by the analytic transformation $\mc A+\mc X\mapsto\mc A+ \widehat{\mc X}$ defined in \eqref{ehf}.
Since it is analytic, there is a positive $c\in\mathbb R$ such that
\[
\|\widehat{\mc X}\|\le c\|\mc X\|\qquad\text{for all }\mc A+\mc X\in N_r(\mc A).
\]
Hence, each pair in $N_r(\mc A)$ is equivalent to a miniversal pair from $N_{cr}(\mc A)$. Thus, if a Kronecker pair $\mc B$ is equivalent to a pair in an arbitrarily small neighborhood of $\mc A$, then $\mc B$ is equivalent to a \emph{miniversal} pair in an arbitrarily small neighborhood of $\mc A$. We use this fact in the proof of Theorem I.
\end{remark}

Miniversal deformations have been constructed for matrices under congruence \cite{f_s} and *congruence \cite{def-sesq},
for pairs of symmetric matrices under congruence \cite{dmy2},
for pairs of skew-symmetric matrices under congruence \cite{dmy1}, and for matrix pairs under contragredient equivalence \cite{gar_ser}.

\section{A direct proof of a weak form of Theorem IV}\label{nfl}

Due to the following theorem, which is a weak form of Theorem IV, it suffices to find immediate successors for all pairs \eqref{hwm} with two direct summands and for all matrix pairs of the form $
\mc D_{k_{1}}(\lambda)\oplus\dots\oplus \mc D_{k_{l}}(\lambda)$.

\begin{theorem}      \label{th}
Let $\mc A=([A_{ij}],[A'_{ij}])$ be the
Kronecker pair \eqref{hwm} partitioned such that the pairs of diagonal blocks $(A_{11},A_{11}'),(A_{22},A_{22}'),\ldots$ are the direct summands \eqref{weq} of \eqref{hwm}.
Write
\[
\mc D_i:=
\mc D_{k_{i1}}(\lambda_1)\oplus\dots\oplus \mc D_{k_{is_{i}}}(\lambda_i), \qquad i=1,\dots,t.
\]
 Then each immediate successor of $\ll\mc A\rr$ is the orbit of some matrix pair obtained from
$\mc A$ by an
arbitrarily small
perturbation of only one pair $(A_{ij},A'_{ij})$ with $i<j$ that is not contained in  $\mc D_1,\dots,\mc D_t$, or of only one pair $(A_{ij},A'_{ij})$ from $\mc D_1,\dots,\mc D_t$.
\end{theorem}

\begin{proof}
We consider the partition of the matrices of $\mc A=(A,A')$ into the blocks $A_{ij}$ and $A'_{ij}$. We also consider the partition of $A$ and $A'$ into the \emph{superblocks} obtained by  joining all strips that correspond to the same eigenvalue. Thus, the diagonal superblocks form the pairs
\begin{equation*}\label{y4r}
\mc L^T_{m_1},\ \dots,\ \mc L^T_{m_{\underline s}},\
\mc D_{1},\ \dots,\ \mc D_t, \
\mc L_{n_1},\ \dots,\ \mc L_{n_{\up s}}.
\end{equation*}

Let $\ll\mc B\rr$ be an immediate successor of $\ll\mc A\rr$. Then there exists a sequence
\begin{equation}\label{seq}
\mc B_1=(B_1,B'_1),\
\mc B_2=(B_2,B'_2),\ \dots
\end{equation}
of pairs
from  $\ll\mc B\rr$ that
converges to $\mc A=(A,A')$.
All matrix pairs close enough
to $\mc A$ are
reduced to the miniversal form \eqref{aaa}
by a smooth equivalence
transformation that preserves $\mc A$. Hence, all pairs
\eqref{seq} can be taken in the miniversal form \eqref{aaa}, which is
\emph{upper superblock triangular.}

We say
that a block (superblock) of $B_i$ or $B'_i$ in \eqref{seq} is
\emph{perturbed} if it
differs from the
corresponding block (superblock) of
$A$ or $A'$.
\bigskip

\noindent
\emph{Case
1: There are
infinite many
pairs \eqref{seq}, in
which at least one
upper diagonal
superblock is
perturbed.}

Then there is a
partition
\begin{equation}\label{iod}
\mc A= \left(
\begin{bmatrix}
  M & O \\
  0  & N
\end{bmatrix},
\begin{bmatrix}
  M' & O' \\
  0  & N'
  \end{bmatrix}\right)\quad\text{($O$ and $O'$ are zero)}
\end{equation}
that is coarser than the partition into superblocks, with the property:
$O$ or $O'$ is
perturbed infinitely
many times in the
sequence \eqref{seq}.
We can suppose that
$O$ or $O'$ is
perturbed in \emph{each}
pair \eqref{seq}.

Let $m\times m'$ be
the size of $(M,M')$.
Partition
\[
\mc B_i= \left(
\begin{bmatrix}
  M_i & O_i \\
  0  & N_i
\end{bmatrix},
\begin{bmatrix}
  M'_i & O'_i \\
  0  & N'_i
  \end{bmatrix}\right)
\]
conformally with \eqref{iod}, and write $\xi_i
:= (\| O_i \| + \|
O'_i \|)^{-1} $, in which $\|\cdot\|$ is the Frobenius matrix norm \eqref{nor}. Define the equivalent
pair
\begin{align*}
\widehat{\mc B}_i:=&
\begin{bmatrix}
  I_m & 0 \\
  0  & \xi_i^{-1}I
\end{bmatrix}
\mc B_i
\begin{bmatrix}
  I_{m'} & 0 \\
  0  &  \xi_iI
\end{bmatrix}
=
 {
\left( \begin{bmatrix}
  M_i & \xi_i O_i \\
  0  & N_i
\end{bmatrix},
\begin{bmatrix}
  M'_i & \xi_i O'_i \\
  0  & N'_i \\
\end{bmatrix}
  \right)}\in \ll{\cal B}\rr.
\end{align*}
Then $\| \xi_iO_i
\| + \| \xi_iO'_i
\|=1$, and so the set
of matrix pairs
$(\xi_iO_i,\xi_iO'_i)$
is compact. Chose a
fundamental
subsequence
$(\xi_{i_{k}}O_{i_{k}},
\xi_{i_{k}}O'_{i_{k}})$
and denote
its limit by $(Q,Q')$. Consider
the pair
\begin{equation*}\label{lso}
{\mc X}
:= \left(
\begin{bmatrix}
  M & Q \\
  0  & N
\end{bmatrix},
\begin{bmatrix}
  M' & Q' \\
  0  &  N'
\end{bmatrix}
\right).
\end{equation*}
We have $\ll\mc B\rr
\ge \ll\mc X\rr$
since all $\widehat{\mc B}_{i_k}
\in\ll\mc B\rr$ and
$\widehat{\mc B}_{i_k}
\to
\mc X$
as $k\to\infty$.

Make additional partitions of
$\mc X$ into blocks
conformally to the partition of
$\mc A=([A_{ij}],[A'_{ij}])$ in the theorem. Choose in
$(Q,Q')$ the nonzero pair
$(X,X')$ of conformal
blocks $X$ and $X'$ such that all columns of $Q$ to the
left of $X$ and all blocks of $Q$
exactly under $X$ are zero, and
all columns of $Q'$ to
the left of $X'$ and all blocks  of $Q'$
exactly under $X'$ are
zero:
\[\setlength{\arraycolsep}{2.5pt}
\mc X=
\left(\left[\begin{array}{ccc|ccc}
 M_1&&0&0 & * & * \\&M_2&&0 & X & *
 \\0&&M_3&0 & 0 & *
               \\\hline
 &&&N_1&&0
 \\&0&&&N_2&
 \\&&&0&&N_3
\end{array}
\right],
\left[\begin{array}{ccc|ccc}
 M'_1&&0&0 & * & * \\&M'_2&&0 & X' & *
 \\0&&M'_3&0 & 0 & *
               \\\hline
 &&&N'_1&&0
 \\&0&&&N'_2&
 \\&&&0&&N'_3
\end{array}
\right]\right).
\]
Write
\begin{equation*}\label{bbv1}
\begin{split}
{\mc Y}&=
\left(\left[\begin{array}{c|c}
 M&Y\\\hline 0&N
\end{array}\right],
\left[\begin{array}{c|c}
 M'&Y'\\\hline 0&N'
\end{array}\right]\right)
         \\&:=
\setlength{\arraycolsep}{3pt}
\left(\left[\begin{array}{ccc|ccc}
 M_1&&0&0 &0& 0 \\&M_2&&0 & X &0
 \\0&&M_3&0 & 0 &0
               \\\hline
 &&&N_1&&0
 \\&0&&&N_2&
 \\&&&0&&N_3
\end{array}
\right],
\left[\begin{array}{ccc|ccc}
 \setlength{\arraycolsep}{1pt}
 M'_1&&0&0 &0&0 \\&M'_2&&0 & X' &0
 \\0&&M'_3&0 & 0 &0
               \\\hline
 &&&N'_1&&0
 \\&0&&&N'_2&
 \\&&&0&&N'_3
\end{array}
\right]\right).
\end{split}
\end{equation*}
Then
\[
(I_a\oplus
\varepsilon^{-1}
I\oplus
\varepsilon^{-2}I_c)
\mc X
( I_b\oplus
\varepsilon I\oplus
\varepsilon^{2}
I_d)
\   \xrightarrow{\text{
as } \varepsilon  \to 0\ }
\ \mc Y,
\]
in which $a\times b$ is the size of $(M_1,M_1')$ and $c\times d$ is the size of $(N_3,N_3')$.
This implies that $\ll\mc X\rr\ge \ll\mc
Y\rr$.
Since
\begin{equation*}\label{kut1}
\mc Y_{\varepsilon}:=
\begin{bmatrix}
  I_{m} & 0 \\
  0  & \varepsilon^{-1}I
\end{bmatrix}
\mc Y
\begin{bmatrix}
  I_{m'} & 0 \\
  0  & \varepsilon I \\
\end{bmatrix}
=
\left(\mat{M&\varepsilon Y\\ 0&N},
\mat{
 M'&\varepsilon Y'\\ 0&N'}\right)
\   \xrightarrow{\text{\,as\,} \varepsilon  \to 0\ }
\
\mc A,
\end{equation*}
we have that $\ll\mc Y\rr \ge \ll\mc A\rr$. Therefore, $\ll\mc B\rr
\ge\ll\mc
X\rr\ge \ll\mc Y\rr \ge \ll\mc A\rr$.

In order to prove that $\ll\mc Y\rr$ is a desired pair,
it suffices to prove that $\ll\mc Y\rr \ne \ll\mc A\rr$ (which implies $\ll\mc B\rr
=\ll\mc Y\rr > \ll\mc A\rr$ because $\ll\mc B\rr$ is an immediate successor of $\ll\mc A\rr$).

On the contrary, suppose that $\ll\mc Y\rr = \ll\mc A\rr$. Since $\mc Y_{\varepsilon}\sim\mc Y$, $\mc Y_{\varepsilon}\in\ll\mc A\rr$ for each $\varepsilon $. Hence there exist nonsingular matrices, which we take in the form $I+R_{\varepsilon}$ and $I+S_{\varepsilon}$, such that
\begin{equation*}\label{uth}
\mc Y_{\varepsilon}=(I+R_{\varepsilon})\mc A(I+S_{\varepsilon})=\mc A+R_{\varepsilon}\mc A
+\mc AS_{\varepsilon}
+R_{\varepsilon}\mc AS_{\varepsilon}.
\end{equation*}
By Lipschitz's property for matrix pairs (see \cite{rod} or \cite{ala}), we can chose the matrices $R_{\varepsilon}$, $S_{\varepsilon}$ and a positive constant $c\in\mathbb R$ such that
\begin{equation}\label{msx}
\|R_{\varepsilon}\|<\varepsilon c,\qquad\|S_{\varepsilon}\|<\varepsilon c
\end{equation}
for all $\varepsilon$, in which $\|\cdot\|$ is the Frobenius matrix norm \eqref{nor}.

The pair $\mc Y_{\varepsilon}$ is in the miniversal form \eqref{aaa} for \eqref{iod} since all nonzero entries of $Q$ and $Q'$ are at the places of some stars.
By the construction of the miniversal deformation in \cite[Theorem 4.1]{gar_ser},
\begin{equation}\label{ajs}
\Delta \mc Y_{\varepsilon}:=\mc Y_{\varepsilon}-\mc A=
\varepsilon\left(\left[\begin{array}{c|c}
 0&\!\!\rule{0pt}{20pt}
 \begin{smallmatrix}\rule{0pt}{3pt}0&0&0
 \\0& X&0\\0&0&0\\[3pt]
 \end{smallmatrix}\!\!\!\\\hline 0&0
\end{array}\right],
\left[\begin{array}{c|c}
 0&\!\!\rule{0pt}{20pt}
 \begin{smallmatrix}\rule{0pt}{3pt}0&0&0
 \\0& X'&0\\0&0&0\\[3pt]
 \end{smallmatrix}\!\!\!\\\hline 0&0
\end{array}\right]\right)
=R_{\varepsilon}\mc A
+\mc AS_{\varepsilon}
+R_{\varepsilon}\mc AS_{\varepsilon}
\end{equation}
does not belong to the space
\[
\mathbb T:=\{R\mc A+\mc AS\,|\,R\text{ and $S$ are nonsingular matrices}\}
\]
(which is the tangent space at $\mc A$ to the orbit of $\mc A$).
Thus,
\begin{equation*}\label{fgq}
d_{\varepsilon}:=\min\big\{\|\mc Y_{\varepsilon}-\mc A-R\mc A-\mc AS\|\,\big|\,
\text{$R$ and $S$ are square matrices}\big\}\ne 0.
\end{equation*}
(which is the distance from $\mc Y_{\varepsilon}$ to the affine space $
\{\mc A+R\mc A+\mc AS\,|\,R,S\}$).

Let $R'$ and $S'$ be such that
\[
d_1=\|\mc Y_1-\mc A-R'\mc A-\mc AS'\|=
\|\Delta\mc Y_1-R'\mc A-\mc AS'\|.
\]
By \eqref{ajs}, $\Delta\mc Y_{\varepsilon}=\varepsilon\Delta\mc Y_1$,
and so
$\varepsilon d_{1}=
\|\Delta\mc Y_{\varepsilon}-(\varepsilon R')\mc A-\mc A(\varepsilon S')\|=d_{\varepsilon}.$ By \eqref{msx},
\[
\varepsilon d_1\le \|\Delta\mc Y_{\varepsilon}-R_{\varepsilon}\mc A-\mc AS_{\varepsilon}\|=\|R_{\varepsilon}\mc AS_{\varepsilon}\|\le
\|R_{\varepsilon}\|\|\mc A\|\|S_{\varepsilon}\|\le \varepsilon^2 c^2\|\mc A\|.
\]
This leads to a contradiction since $\varepsilon d_1\le \varepsilon^2 c^2\|\mc A\|$ does not hold for a sufficiently small $\varepsilon $.

\bigskip

\noindent\emph{Case
2: There is only a finite number of
pairs \eqref{seq} in
which at least one
upper diagonal
superblock is perturbed.}

Let $\mc A^{(1)}, \mc A^{(2)},\dots$ be the pairs of diagonal superblocks of $\mc A$, then $\mc A=\mc A^{(1)}\oplus \mc A^{(2)} \oplus \cdots$.
We can suppose that
all upper diagonal
superblocks are not
perturbed, and so $
\mc B_i:=\mc B^{(1)}_i\oplus \mc B^{(2)}_i \oplus \cdots,$
in which $\mc B^{(1)}_i, \mc B^{(2)}_i,\dots$ are the pairs of perturbed diagonal superblocks
of $\mc B_i$ in \eqref{seq}.

Since all $\mc B_i\sim \mc B$, we can suppose that $\mc B_1^{(l)}\sim\mc B_2^{(l)}\sim\cdots$ for each $l$. Since $\mc A\nsim\mc B$, $\mc A^{(l)}\nsim\mc B_1^{(l)}\sim\mc B_2^{(l)}\sim\cdots$ for some $l$.
Then all
\[
\mc C_i:=\mc A^{(1)}\oplus\dots \oplus \mc A^{(l-1)}\oplus \mc B^{(l)}_i \oplus \mc A^{(l+1)}\oplus\cdots
\]
are equivalent and their orbit $\ll{\mc C}_1\rr> \ll\mc A\rr$. Moreover, $\ll\mc B\rr\ge\ll{\mc C}_1\rr$ because
\[
\mc B^{(1)}_i\oplus\dots \oplus \mc B^{(l-1)}_i\oplus \mc B^{(l)}_1 \oplus \mc B^{(l+1)}_i\oplus\cdots\
\ \xrightarrow{\text{
as } i  \to \infty\ } \ \mc C_1.
\]
Since there is no
intermediate orbit
between $\ll{\cal A}\rr$ and
$\ll{\cal B}\rr$, we have that
$\ll{\cal B}\rr =
\ll{\cal C}_1\rr$.
\end{proof}

\section{Perturbations of direct sums of two indecomposable Kronecker pairs}\label{pairs}

\subsection{Perturbations of $\mc L^T_m \oplus\mc L^T_n$}

\begin{theorem}
\label{case1}
{\rm(a)}
The set of Kronecker
canonical forms of all pairs in a sufficiently small neighborhood of
\begin{equation}\label{(1)}
\mc L^T_m \oplus
\mc L^T_n,\qquad m\le n
\end{equation}
consists of the pairs
\begin{equation}\label{u5g}
\mc L^T_{m+r} \oplus
\mc L^T_{n-r},\qquad m+r\le n-r,\ r\ge 0.
\end{equation}

{\rm(b)} Each pair \eqref{u5g} with $r>0$ is equivalent to a pair of the form
\begin{equation}\label{vbf}
\left(
\mat{L^T_m&0\\0&L^T_n},
\mat{R^T_m&\D_r(\varepsilon)
\\0&R^T_n}
\right)
\end{equation}
(which is obtained by an arbitrarily small perturbation of \eqref{(1)}),
in which $\D_r(\varepsilon )$ is defined in \eqref{rdpd} and $\varepsilon$ is an arbitrary nonzero complex number.
\end{theorem}

\begin{lemma}\label{kkk}
 Each pair of $n\times (n-1)$ matrices of the form
\begin{equation}\label{oui}
\left(\begin{bmatrix}
         1&*&&*\\
         0&1&\ddots&\\
         &0&\ddots&*\\
         &&\ddots&1\\0&&&0
         \end{bmatrix},\
\begin{bmatrix}
         *&*&&*\\
         1&*&\ddots&\\
         &1&\ddots&*\\
         &&\ddots&*\\0&&&1
         \end{bmatrix}\right)
\end{equation}
is reduced to $\mc L^T_n$ by simultaneous additions of columns from left to right and simultaneous additions of rows from the bottom to up.
\end{lemma}

\begin{proof}
Consider the subpair $\cal P$ of \eqref{oui} obtained by removing the last row and last column in the matrices of the pair \eqref{oui}. Reasoning by induction on $n$, we suppose that the subpair $\cal P$ is reduced to $\mc L^T_{n-1}$ by simultaneous additions of columns of its matrices from left to right and simultaneous additions of rows from the bottom to up. We obtain \eqref{oui} in which all stars are zero except for some stars of the last columns.
We make zero the stars of the last column in the first matrix by adding the other columns simultaneously in both matrices; then we make zero the stars of the last column in the second matrix by adding the last row.
\end{proof}

\begin{proof}[Proof of Theorem {\rm\ref{case1}}]
(a) \
By Theorem \ref{first}, there is a neighborhood of \eqref{(1)}, in which all pairs are equivalent to pairs of the form
\begin{equation}\label{(2)}
(C,D):=\left(
\left[\begin{array}{c|c}
  \\[-7pt]L^T_m & \quad 0\quad \\[7pt]
  \hline
  0 & \rule{0pt}{13pt}L^T_n
\end{array}\right],
\left[\begin{array}{c|c}
   R^T_m & \begin{matrix}
  \alpha _1\,\dots\, \alpha _{n-1}\\0
  \end{matrix} \\\hline
  0 & \rule{0pt}{13pt}R^T_n
\end{array}\right]
\right),\quad\text{all }\alpha _i\in\mathbb C,
\end{equation}
in which the last $m$ entries in the sequence $\alpha _1,\dots, \alpha _{n-1}$ are zero.
It is sufficient to prove that $(C,D)$ is equivalent to a pair of the form \eqref{u5g}.

We can suppose that not all $\alpha _1,\dots,\alpha _{n-1}$ are zero (otherwise,
$(C,D)$ is the pair \eqref{(1)}).
Let $\alpha_s$ be the first nonzero entry. Then
\begin{equation}\label{nos}
1\le s<n-m \quad\text{if }m\ne n.
\end{equation}
Let us reduce $(C,D)$ by simultaneous elementary transformations to the form \eqref{u5g}. We usually specify only  transformations with one of the matrices $C$ and $D$ which means that we make the same transformations with the other matrix.
We divide the first horizontal strips of $C$ and $D$ by $\alpha_s$, then multiply the  first vertical strips by $\alpha_s$, and obtain\\
$(C,D)=\left(\left[
\begin{MAT}(@){c3c}
C_{11}&C_{12}\\3
C_{21}&C_{22}\\
 \end{MAT}\right],\,
\left[
\begin{MAT}(@){c3c}
D_{11}&D_{12}\\3
D_{21}&D_{22}\\
 \end{MAT}\right]
\right)$\\[-1mm]
\begin{equation}\label{uuu}
=\left(\
\begin{MAT}(@)[0.9pt]{ccccccccccccc}
\mathit{\scriptstyle 1} &\phantom{0_0}&
\mathit{\scriptstyle\hspace{-7pt} m-1} &\mathit{\scriptstyle 1}&&&{\scriptstyle\!s\!} &&&
 \mathit{\scriptstyle\!\!\!\!\!\!s+m-1\!\!\!\!\!\!}
  &&&\\
 \,1&& &&&& &&& &&&\\
 \,0&\ddotc& &&&& &&& &&&\\
 &\ddotc&1\! &&&& &&& &&&\\
 \phantom{0_0}&&0\! &&&& &&& &&&\\
 && &1&&& &&& &&&\\
 && &0&\ddotc&& &&& &&&\\
 && &&\ddotc&1& &&& &&&\\
 &\phantom{0_0}& &&&0&1 &&& &&&\\
 \,\ze&& &&&&0 &\,1&& &&&\\
 0&\ddotc& &&&& &\,0&\ddotc& &&&\\
 &\ddotc&\ze &&&& &&\ddotc&1\!\! &&&\\
 &&0 &&&& &\phantom{0_0}&&0\!\! &\,1&&\\
 && &&&& &&& &0\,&\ddotc&\\
 && &&&& &&& &&\ddotc&1\\
 &\phantom{0_0}& &&&& &&& &&&0
\addpath{(3,11,3)rrrrrrrrrrr}
\addpath{(3,11,4)lll}
 \addpath{(0,11,4)rrruuuullldddd}
 \addpath{(0,3,4)rrruuuulll}
 \addpath{(7,3,4)rrruuuullldddd}
 \addpath{(0,0,4)rrrrrrrrrrrrr%
 uuuuuuuuuuuuuuulllllllllllll}
 \addpath{(0,0,4)uuuuuuuuuuuuuuu}
 \addpath{(7,0,.)uuuuuuuuuuuuuuu}
\addpath{(10,0,.)uuuuuuuuuuuuuuu}
\addpath{(3,3,.)rrrr}\addpath{(10,3,.)rrrr}
\addpath{(3,7,.)rrrr}\addpath{(10,7,.)rrrr}
\addpath{(3,0,3)uuu}\addpath{(3,7,3)uuuu}
\addpath{(3,15,3)u} \\ \end{MAT}\ ,\
 \begin{MAT}(@)[0.9pt]{cccccccccccccl}
\mathit{\scriptstyle 1} &\phantom{0_0}&\mathit{\scriptstyle\hspace{-7pt} m-1} &\mathit{\scriptstyle 1}&&&{\scriptstyle\!s\!} &&&
 \mathit{\scriptstyle\!\!\!\!\!\!s+m-1\!\!\!\!\!\!}
  &&&&\\
 \,0&& &\,0&...&0& 1&*&\!...\!\!& *\!&*&\!...\!&* &
 \mathit{\,\scriptstyle 1}\\
 \,1&\ddotc& &&&& &&& &&& &\\
 &\ddotc&0\! &&&& &&& &&& &\\
 &\phantom{0_0}&1\! &&&& &&& &&& &
 {\,\scriptstyle m}
 \\
 && &0&&& &&& &&& &\mathit{\,\scriptstyle 1}\\
 && &1&\ddotc&& &&& &&& &\\
 && &&\ddotc&0& &&& &&& &\\
 &\phantom{0_0}& &&&1&0 &&& &&& &
{\,\scriptstyle s}\\
 0&& &&&&1 &\,0&& &&& &
 \mathit{\,\scriptstyle s+1\hspace{-27pt}}\\
 \,\ze&\ddotc& &&&& &\,1&\ddotc& &&& &\\
 &\ddotc&0 &&&& &&\ddotc&0\!\! &&& &\\
 &&\ze &&&& &&\phantom{0_0}&1\!\! &\,0&& &
 {\,\scriptstyle s+m\hspace{-27pt}}\\
 && &&&& &&& &1&\ddotc& &\\
 && &&&& &&& &&\ddotc&0 &\\
 &\phantom{0_0}& &&&& &&& &&&1 &
\addpath{(3,11,3)rrrrrrrrrrr}
\addpath{(3,11,4)lll}
 \addpath{(0,11,4)rrruuuullldddd}
 \addpath{(0,3,4)rrruuuulll}
 \addpath{(7,3,4)rrruuuullldddd}
 \addpath{(0,0,4)rrrrrrrrrrrrr%
 uuuuuuuuuuuuuuulllllllllllll}
 \addpath{(0,0,4)uuuuuuuuuuuuuuu}
 \addpath{(7,0,.)uuuuuuuuuuuuuuu}
\addpath{(10,0,.)uuuuuuuuuuuuuuu}
\addpath{(3,3,.)rrrr}\addpath{(10,3,.)rrr}
\addpath{(3,7,.)rrrr}\addpath{(10,7,.)rrr}
\addpath{(3,0,3)uuu}\addpath{(3,7,3)uuuu}
\addpath{(3,15,3)u}
      \\
 \end{MAT}\hspace{13pt}
\right)
\end{equation}
with $\alpha_s=1$. We reduce $(C,D)$ by the following simultaneous elementary transformations in order to make zero the entry ``1'' under $\alpha_s$ (the zero entries in \eqref{uuu} that first are transformed to $-1$ and then are restored to $0$ are denoted by $\ze$):
\begin{itemize}
  \item The strip $[D_{11}\; D_{12}]$ is subtracted from the substrip formed by rows $s+1, s+2,\dots, s+m$ in the strip $[D_{21}\; D_{22}]$.  Thus, the block $(1,1)$ is subtracted from the rectangle in the block $(2,1)$ (see \eqref{uuu}).

\item  Then the substrip formed by columns $s+1,\dots,s+m-1$ in $\matt{D_{12}\\D_{22}}$ is added to $\matt{D_{11}\\D_{21}}$.  Thus, the rectangle in the block $(2,2)$ is added to the rectangle in the block $(2,1)$ restoring it.

\end{itemize}
We obtain
\begin{equation}\label{gri}
(C,D)= \left(
 \left[\begin{MAT}(b){ccc3cccc2ccc}
 1&&& &&& &&&\\
 0&\ddotc&& &&& &&&\\
 &\ddotc&1& &&& &&&\\
 &&0& &&& &&&\\3
 &&& 1&&& &&&\\
 &&& &\ddotc&& &&&\\
 &&& &&1& &&&\\
 &&& &&&1 &&&\\2
 &&& &&& &1&&\\
 &&& &&& &0&\ddotc&\\
 &&& &&& &&\ddotc&1\\
 &&& &&& &&&0\\
 \end{MAT}
 \right],\
 \left[\begin{MAT}(b){ccc3cccc2ccc}
 *&\!...\!\!&*& 0&\!...\!\!&0&1 &*&\!...\!\!&*\\
 1&&0& &&& &&&\\
 &\ddotc&& &&& &&&\\
 &&1& &&& &&&\\3
 &&& 0&&& &&&\\
 &&& 1&\ddotc&& &&&\\
 &&& &\ddotc&0& &&&\\
 &&& &&1&0 &&&\\2
 *&\!...\!\!&*& &&& &*&\!...\!\!&*\\
 &&& &&& &1&&0\\
 &&& &&& &&\ddotc&\\
 &&& &&& &&&1\\
 \end{MAT}
 \right]\right),
\end{equation}
in which the stars denote complex numbers.
Interchange
the first and second vertical strips, then the first and second horizontal strips,
and obtain
\begin{align}\nonumber
(C,D)&=\left(\left[
\begin{MAT}(@){c3c3c}
C_{11}&C_{12}&C_{13}\\3
C_{21}&C_{22}&C_{23}\\3
C_{31}&C_{32}&C_{33}\\
 \end{MAT}\right],\,
\left[
\begin{MAT}(@){c3c3c}
D_{11}&D_{12}&D_{13}\\3
D_{21}&D_{22}&D_{23}\\3
D_{31}&D_{32}&D_{33}\\
 \end{MAT}\right]
\right)
          \\ \label{(3)}
&= \left(
 \left[\begin{MAT}(@)[1.5pt]{cccc3cccc2ccc}
 1&&&& &&& &&&\\
 0&\ddotc&&& &&& &&&\\
 &\ddotc&1&& &&& &&&\\
 &&0&1& &&& &&&\\3
 &&&0& 1&&& &&&\\
 &&&& 0&1&& &&&\\
 &&&& &0&\ddotc& &&&\\
 &&&& &&\ddotc&1 &&&\\
 &&&& &&&0 &&&\\2
 &&&& 0&*&\!...\!\!&* &1&&\\
 &&&& &0&\ddotc&\vdotc &0&\ddotc&\\
 &&&& &&\ddotc&* &&\ddotc&1\\
 &&&& &&&0 &&&0\\
 \end{MAT}
 \right],\
 \left[\begin{MAT}(b){cccc3cccc2ccc}
 0&&&& &&& &&&\\
 1&\ddotc&&& &&& &&&\\
 &\ddotc&0&& &&& &&&\\
 &&1&0& &&& &&&\\3
 &&&1& *&*&\!...\!\!&* &*&\!...\!\!&*\\
 &&&& 1&&&0 &&&\\
 &&&& &1&& &&&\\
 &&&& &&\ddotc& &&&\\
 &&&& &&&1 &&&\\2    
 &&&& *&*&\!...\!\!&* &*&\!...\!\!&*\\
 &&&& &*&\ddotc&\vdotc &1&&0\\
 &&&& &&\ddotc&* &&\ddotc&\\
 &&&& &&&* &&&1\\
 \end{MAT}
 \right]\right),
\end{align}
in which we replace by stars some zero entries of blocks $C_{32}$ and $D_{32}$.

Using transformations from Lemma \ref{kkk}, we make zero all stars in $D_{33}$; the forms of the other blocks do not change. Make zero row 1 of $D_{32}$ by adding  rows $2, 3,\dots$ of horizontal strip 2 to row 1 of strip 3 simultaneously in $C$ and $D$. Make zero row 1 of $C_{32}$ by adding column 1 of vertical strip 3 simultaneously in $C$ and $D$. Then, adding  rows $3, 4,\dots$ of  strip 2 to the row 2 of strip 3, we make zero row 2 of $D_{32}$. Adding column 2 of vertical strip 3 we make zero row 2 of $C_{32}$, and so on until we obtain \eqref{(3)} in which all stars in  horizontal strips 3 of $C$ and $D$ are zero.

Using Lemma \ref{kkk}, we make zero all stars in $D_{22}$.
Multiplying horizontal strips 2 in $C$ and $D$ by an arbitrarily small number and then dividing vertical strips 2 by the same number, we make the entries of $D_{23}$ arbitrarily small; these transformations do not change the other blocks. We obtain the pair that is equivalent to the initial perturbed pair \eqref{(2)}
and that is obtained from
$\mc L^T_{m+s}\oplus \mc L^T_{n-s}$ by an arbitrarily small perturbation, in which $s$ as in \eqref{uuu} and satisfies \eqref{nos}.
We interchange $\mc L^T_{m+s}$ and $\mc L^T_{n-s}$ if $m+s>n-s$, and reduce the obtained pair by equivalence transformations to its miniversal form
\begin{equation}\label{(2a)}
\left(
\left[\begin{array}{c|c}
  \\[-7pt]L^T_{m'} & \quad 0\quad \\[7pt]
  \hline
  0 & \rule{0pt}{13pt}L^T_{n'}
\end{array}\right],
\left[\begin{array}{c|c}
   R^T_{m'} & \begin{matrix}
  *\dots *\\0
  \end{matrix} \\\hline
  0 & \rule{0pt}{13pt}R^T_{n'}
\end{array}\right]
\right),
\end{equation}
in which the stars are sufficiently small complex numbers. By \eqref{nos},
\[
m<m':=\min(m+s,n-s)\le n':=\max(m+s,n-s).
\]
We repeat this procedure until we obtain a pair
\begin{equation}\label{(2ab)}
\left(
\left[\begin{array}{c|c}
  \\[-7pt]L^T_{m^{(l)}} & \quad 0\quad \\[7pt]
  \hline
  0 & \rule{0pt}{13pt}L^T_{n^{(l)}}
\end{array}\right],
\left[\begin{array}{c|c}
   R^T_{m^{(l)}} & \begin{matrix}
  *\dots *\\0
  \end{matrix} \\\hline
  0 & \rule{0pt}{13pt}R^T_{n^{(l)}}
\end{array}\right]
\right)
\end{equation}
in which all stars are zero, and $m<m^{(l)}\le n^{(l)}$. Thus, \eqref{(2ab)} is of the form \eqref{u5g} with $r>0$.
\medskip

(b)
Let $\mc L^T_{m+r} \oplus
\mc L^T_{n-r}$ be the pair \eqref{u5g} with $r>0$; we must prove that it is equivalent to \eqref{vbf}. We divide the first horizontal strips of \eqref{vbf} by $\varepsilon $, then multiply the  first vertical strips by $\varepsilon $, and obtain the pair \eqref{uuu} in which all stairs are zero.
The obtained pair is reduced as above to \eqref{gri} in which all stairs are zero. This pair is  permutation equivalent to $\mc L^T_{m+r} \oplus
\mc L^T_{n-r}$.
\end{proof}

\subsection{Perturbations of ${\mc L_n \oplus
\mc L_m}$}

\begin{theorem} \label{case1a}
{\rm(a)}
The set of Kronecker
canonical forms of all pairs in a sufficiently small neighborhood of
\begin{equation}\label{(1a)}
\mc L_m \oplus
\mc L_n,\qquad m\le n
\end{equation}
consists of the pairs
\begin{equation}\label{u5ga}
\mc L_{m+r} \oplus
\mc L_{n-r},\qquad m+r\le n-r,\  r\ge 0.
\end{equation}

{\rm(b)} Each pair \eqref{u5ga} with $r>0$ is equivalent to a pair of the form
\begin{equation*}\label{vbfa}
\left(
\mat{L_m&0\\0&L_n},
\mat{R_m&0
\\\D_r(\varepsilon )^T&R_n}
\right)
\end{equation*}
(which is obtained by an arbitrarily small perturbation of \eqref{(1a)}),
in which $\D_r(\varepsilon )$ is defined in \eqref{rdpd} and $\varepsilon$ is an arbitrary nonzero complex number.
\end{theorem}

\begin{proof} This theorem is obtained from Theorem \ref{case1} by matrix transposition.
\end{proof}

\subsection{Perturbations of {${\mc L^T_m\oplus
{\cal D}_n(\lambda)}$}}

\begin{theorem}\label{case4}
The set of Kronecker
canonical forms of all pairs obtained by perturbations of the blocks $(1,2)$ in
\begin{equation}\label{bmx}
\mc L^T_m\oplus {\cal D}_n(\lambda )=
\begin{cases}
\left(
\mat{L^T_m&0\\ 0&I_n},\,
\mat{
R^T_m&0\\ 0&J_n(\lambda )
}
 \right)&\text{if }\lambda \in\mathbb C
                        \\[18pt]
\left(
 \mat{
L^T_m&0\\ 0&J_n(0)
},\,
\mat{
R^T_m&0\\ 0&I_n
} \right)&\text{if }\lambda =\infty
\end{cases}
\end{equation}
consists of the pairs
\begin{equation}\label{cxm}
\mc L^T_{m+r} \oplus \mc D_{n-r}(\lambda ),\quad\text{in which } 0\le r\le n.
\end{equation}

{\rm(b)} Each pair \eqref{cxm} with $r>0$ is equivalent to a pair of the form
\begin{equation}\label{bmx1}
\begin{aligned}
&\left(
 \mat{
L^T_m&0\\ 0&I_n},\,
\mat{
R^T_m&\D_{n-r+1}(\varepsilon)
\\ 0&J_n(\lambda )}
 \right)
&&\text{if }\lambda \in\mathbb C
                             \\
&\left(
 \mat{
L^T_m&\nabla_{n-r+1}(\varepsilon)\\ 0&J_n(0)},\,
\mat{
R^T_m&0\\ 0&I_n}
 \right)
&&\text{if }\lambda =\infty
\end{aligned}
\end{equation}
(which is obtained by an arbitrarily small perturbation of \eqref{bmx}),
in which $\D_r(\varepsilon )$ and $\nabla_r(\varepsilon)$ are defined in \eqref{rdpd} and $\varepsilon$ is an arbitrary nonzero complex number.
\end{theorem}

\begin{proof}
Let $(A,B)$ be the pair \eqref{bmx} with $\lambda =\infty$.
Since
\begin{equation*}\label{xsm}
(R^T_m,L^T_m)=Z_m(L^T_m,R^T_m)Z_{m-1},\qquad Z_p:=\mat{0&&1\\&
 \udots&\\1&&0}\ (p\text{-by-}p) ,
\end{equation*}
$(B,A)$ is equivalent to the pair \eqref{bmx} with $\lambda =0$. Therefore, it suffices to prove the theorem for $\lambda \in\mathbb C$.

Let $(A,B(\lambda ))$ be the pair \eqref{bmx} with $\lambda \in\mathbb C$. Since $(L^T_m, R^T_m-\lambda L^T_m)$ is equivalent to $(L^T_m, R^T_m)$, the pair $(A,B(\lambda )-\lambda A)$ is equivalent to $(A,B(0))$. Therefore, it
suffices to prove the theorem for $\lambda =0$. In the rest of the proof, we set $\lambda=0$.
\medskip

(a)
Let $(C,D)$ be a pair that is obtained from
\eqref{bmx} with $\lambda=0$ by replacing its blocks $(1,2)$ by arbitrary matrices; we must prove that the Kronecker canonical form of $(C,D)$ is \eqref{cxm} for some $r$.

Multiplying the first horizontal strips of $C$ and $D$ by an arbitrarily small number and then dividing the first vertical strips by the same number, we make the entries of the blocks $(1,2)$ arbitrarily small.
Theorem \ref{first} ensures that
$(C,D)$ is reduced by equivalence transformations to the form
\begin{align}\nonumber
(C,D)&=\left(\left[\begin{array}{c|c}
 C_{11}&C_{12}\\\hline C_{21}&C_{22}
\end{array}\right],
\left[\begin{array}{c|c}
 D_{11}&D_{12}\\\hline D_{21}&D_{22}
\end{array}\right]\right)\\ \label{v5s}
&=
\left(
 \left[\begin{array}{ccc|cccc}
1&&   &&&&\\
0&\ddots&   &&&&\\
 &\ddots&1   &&&\text{\raisebox{8pt}[0pt][0pt]{$\!\!\!\!\!\!0$}}&\\
 &&0   &&&&\\ \hline
 &&   &1&&&\\
 &&   &&1&&\\
 &\text{\raisebox{7pt}[0pt][0pt]{$0$}}&   &&&\ddots&\\
 &&   &&&&1\\
 \end{array}
 \right],\,
 \left[\begin{array}{ccc|cccc}
0&&   &\alpha_1&\alpha_2&...&\alpha_n\\
1&\ddots&   &&&&\\
 &\ddots&0   &&&\!\!\!\!\!\!0&\\
 &&1   &&&&\\ \hline
 &&&   0&1&&\\
 &&   &&0&\ddots&\\
 &\text{\raisebox{7pt}[0pt][0pt]{$0$}}&   &&&\ddots&1\\
 &&   &&&&0\\
 \end{array}
 \right]
 \right),
\end{align}
in which $\alpha_1,\dots,\alpha_n$
are arbitrarily small.

Each matrix that commutes with $J_n(0)$ has the form
\begin{equation*}\label{jc4}
K:=\begin{bmatrix}
\kappa_1&\kappa_2&\ddots&\kappa_n
\\&\kappa_1&\ddots&\ddots\\
&&\ddots&\kappa_2
\\0&&&\kappa_1
\end{bmatrix},\qquad \kappa_1,\dots,\kappa_n\in\mathbb C.
\end{equation*}
The equivalence transformation
\[
(I_m\oplus K^{-1})(C,D)(I_{m-1}\oplus K),\qquad c_1\ne 0
\]
replaces $(\alpha_1,\dots,\alpha_n)$ by
\begin{equation}\label{ff0}
(\alpha_1,\dots,\alpha_n) K=(\alpha_1\kappa_1,\,
\alpha_1\kappa_2+\alpha_2\kappa_1,\,\dots,\,
\alpha_1\kappa_n+\dots+
\alpha_n\kappa_1)
\end{equation}
and does not change the other entries of $C$ and $D$. Let $\alpha_s$ be the first nonzero entry in $(\alpha_1,\dots,\alpha_n)$.
Using transformations \eqref{ff0}, we make $(\alpha_1,\dots,\alpha_n)=
(0,\dots,0,1,0,\dots,0)$ with ``1'' at the position $s$.

Let first $s\ge 2$.
The ``1'' under $\alpha_s=1$ is the $(s-1,s)$th entry of the block $D_{22}$ (see \eqref{v5s}).
We make zero this entry of $D_{22}$ by the following elementary transformations:

\begin{itemize}
  \item \emph{Case 1: $m<s$.}
We subtract the rows $1,2,\dots,m$ of the first horizontal strip from the rows $s-1,s-2,\dots,s-m$
of the second horizontal strip, respectively, in $C$ and $D$.
Then we add the columns $s-1,s-2,\dots,s-m+1$ of the second vertical strip to the columns $1,2, \dots,m-1$ of the first vertical strip in $C$ and $D$. For example, if $m=3$, $n=6$, and $s=5$, then
\[
(C,D)= \left(\,\begin{MAT}(e)[1.5pt]{cccccccc}
1&0&&&&&&\\
0&1&&&&&&\\
0&0&&&&&&\\3
&&1&&&&&\\
0&0&&1&0&0&&\\
0&\ze&&&1&0&&\\
\ze&0&&&0&1&&\\
&&&&&&1&\\
&&&&&&&1
\addpath{(0,0,4)uuuuuuuuurrrrrrrrdddddddddlllllllll}
\addpath{(0,2,4)rruuullddd}
\addpath{(0,6,4)rruuullddd}
\addpath{(4,2,4)rruuullddd}
\addpath{(4,0,.)uull}
\addpath{(6,0,.)uurr}
\addpath{(2,5,.)rruuuu}
\addpath{(8,5,.)lluuuu}
\addpath{(2,0,3)uu}\addpath{(2,5,3)u}
\\
 \end{MAT}\,,\,
\begin{MAT}(e)[1.5pt]{cccccccc}
0&0&&&&&1&\\
1&0&&&&&&\\
0&1&&&&&&\\3
&&0&1&&&&\\
0&\ze&&0&1&0&&\\
\ze&0&&&0&1&&\\
0&0&&&0&0&1&\\
&&&&&&0&1\\
&&&&&&&0
\addpath{(0,0,4)uuuuuuuuurrrrrrrrdddddddddlllllllll}
\addpath{(0,2,4)rruuullddd}
\addpath{(0,6,4)rruuullddd}
\addpath{(4,2,4)rruuullddd}
\addpath{(4,0,.)uull}
\addpath{(6,0,.)uurr}
\addpath{(2,5,.)rruuuu}
\addpath{(8,5,.)lluuuu}
\addpath{(2,0,3)uu}\addpath{(2,5,3)u}
\\
 \end{MAT}\,
\right);
\]
the zero entries that are transformed to $-1$ and then are restored to $0$ are denoted by $\ze$.

  \item \emph{Case 2:
$m\ge s$.}
We subtract the rows $1,2,\dots,s-1$ of the first horizontal strip from the rows $s-1,s-2,\dots,1$
of the second horizontal strip, respectively, in $C$ and $D$.
Then we add the columns $s-1,s-2,\dots,1$ of the second vertical strip to the columns $1,2,\dots,s-1$ of the first vertical strip in $C$ and $D$. For example,
if $m=5$, $n=4$, and $s=3$, then
\[
(C,D)= \left(\,\begin{MAT}(e)[1.5pt]{cccc3cccc}
1&0&&&&&&\\
0&1&&&&&&\\
&0&1&&&&&\\
&&0&1&&&&\\
&&&0&&&&\\3
0&\ze&&&1&0&&\\
\ze&0&&&0&1&&\\
&&&&&&1&\\
&&&&&&&1
\addpath{(0,0,4)uuuuuuuuurrrrrrrrdddddddddlllllllll}
\addpath{(0,2,4)rruulldd}
\addpath{(0,7,4)rruulldd}
\addpath{(4,2,4)rruulldd}
\addpath{(2,0,.)uurr}
\addpath{(6,0,.)uurr}
\addpath{(2,4,.)uuu}
\addpath{(6,4,.)uuuuu}
\addpath{(2,7,.)rrrrrr}
\\
 \end{MAT}\,,\,
\begin{MAT}(e)[1.5pt]{cccc3cccc}
0&0&&&&&1&\\
1&0&&&&&&\\
&1&0&&&&&\\
&&1&0&&&&\\
&&&1&&&&\\3
\ze&0&&&0&1&&\\
0&0&&&0&0&1&\\
&&&&&&0&1\\
&&&&&&&0
\addpath{(0,0,4)uuuuuuuuurrrrrrrrdddddddddlllllllll}
\addpath{(0,2,4)rruulldd}
\addpath{(0,7,4)rruulldd}
\addpath{(4,2,4)rruulldd}
\addpath{(2,0,.)uurr}
\addpath{(6,0,.)uurr}
\addpath{(2,4,.)uuu}
\addpath{(6,4,.)uuuuu}
\addpath{(2,7,.)rrrrrr}
\\
 \end{MAT}\,
\right)\,.
\]
\end{itemize}

Let now $s=1$. The pair $(C,D)$ is permutation equivalent to
$(L^T_{m+n},R^T_{m+n})$, which is a pair of the form \eqref{cxm}.

Therefore,
for each $s$ the pair
$(C,D)$ is reduced to the pair that is obtained from \eqref{v5s} by replacing $(\alpha_1,\dots,\alpha_n)$ by
$(0,\dots,0,1,0,\dots,0)$ and the entry ``1'' under $\alpha _s$ by 0.
This pair is permutation equivalent to
$(L^T_{m+n-s+1}, R^T_{m+n-s+1}) \oplus
( I_{s-1}, J_{s-1}(0 ))$, which is a pair of the form \eqref{cxm}.
\medskip

(b)
Let $(E,F):=\mc L^T_{m+r} \oplus \mc D_{n-r}(0)$ with $0< r\le n$
be the pair \eqref{cxm} with $\lambda =0$; we must prove that it is equivalent to
\eqref{bmx1}. The pair \eqref{bmx1}  with $\lambda =0$ is the pair \eqref{v5s} in which
$(\alpha_1,\dots,\alpha_n)
=(0,\dots,0,\varepsilon ,0,\dots,0)$ with $\varepsilon\ne 0$ at the place $s:= n-r+1$. Reasoning as in part (a), we reduce it to a pair that is permutation equivalent to $(E,F)$.
\end{proof}

\subsection{Perturbations of
{${\mc L_m\oplus {\cal D}_n(\lambda)}$
}}

\begin{theorem}\label{case4a}
{\rm(a)}
The set of Kronecker
canonical forms of all pairs obtained by perturbations of the blocks $(2,1)$ in
\begin{equation}\label{bmxa}
\mc L_m\oplus {\cal D}_n(\lambda )
\end{equation}
consists of the pairs
\begin{equation}\label{cxma}
\mc L_{m+r} \oplus \mc D_{n-r}(\lambda ),\quad\text{in which } 0\le r\le n.
\end{equation}

{\rm(b)} Each pair \eqref{cxma} with $r>0$ is equivalent to a pair of the form
\[\begin{aligned}
&\left(
 \mat{
L_m&0\\ 0&I_n},\,
\mat{
R_m&0
\\ \D_{r}(\varepsilon)^T&J_n(\lambda )}
 \right)
&&\text{if }\lambda \in\mathbb C
                             \\
&\left(
 \mat{
L_m&0\\ 0&J_n(0)},\,
\mat{
R_m&0\\ \nabla_{r}(\varepsilon)^T&I_n}
 \right)
&&\text{if }\lambda =\infty
\end{aligned}\]
(which is obtained by an arbitrarily small perturbation of \eqref{bmxa}),
in which $\varepsilon$ is an arbitrary nonzero complex number.
\end{theorem}

\begin{proof} The mapping
\[
A\mapsto\mat{I_{m-1}&0\\0&Z_n}A^T
\mat{I_{m}&0\\0&Z_n},\qquad Z_n:=\mat{0&&1\\&
\udots&\\1&&0}\ (n\text{-by-}n)
\]
transforms the matrices from Theorem \ref{case4} to the matrices from Theorem \ref{case4a}.
\end{proof}

\subsection{Perturbations of ${\mc D_m(\lambda)\oplus\mc D_n(\lambda)}$}

\begin{theorem}    \label{kurf}
{\rm(a)}
If a Kronecker pair $\mc K$ is equivalent to a pair in an arbitrarily small neighborhood of
\begin{equation}\label{moeg}
\mc D_m(\lambda)\oplus\mc D_n(\lambda),\qquad m\le n,
\end{equation}
then $\mc K$ has the form
\begin{equation}\label{cfqf}
\mc D_{m-r}(\lambda)
\oplus \mc D_{n+r}(\lambda),\qquad 0\le r\le m.
\end{equation}

{\rm(b)} Each pair \eqref{cfqf} with $r>0$ is equivalent to a pair of the form
\begin{equation*}\label{tgth}
  \begin{aligned}
&\left(I_{m+n},\mat{J_m(\lambda )&\D_r(\varepsilon)^T\\0&J_n(\lambda )}\right)
&&\text{if }\lambda \in\mathbb C
                        \\
&\left(\mat{J_m(0)&\D_r(\varepsilon )^T\\0&J_n(0)},I_{m+n}\right)
&&\text{if }\lambda =\infty
  \end{aligned}
\end{equation*}
(which is obtained by an arbitrarily small perturbation of \eqref{moeg}), in which $\varepsilon$ is an arbitrary nonzero complex number.
\end{theorem}

\begin{proof}
This theorem follows from Theorem \ref{kur} by the reasons that are given at the beginning of Section \ref{jord}.
\end{proof}

\subsection{Perturbations of ${\mc L^T_m \oplus \mc L_n}$}

\begin{theorem}\label{case2}
{\rm(a)}
The set of Kronecker
canonical forms of all pairs in a sufficiently small neighborhood of
\begin{equation}\label{(a1)}
\mc L^T_m \oplus \mc L_n
\end{equation}
consists of the pairs
\eqref{(a1)} and
\begin{equation}\label{a1def}
{\cal D}_{r_1} (\lambda_1)\oplus\dots\oplus {\cal D}_{r_t}(\lambda_t),\qquad r_1+\dots +r_t=m+n-1,
\end{equation}
with distinct eigenvalues $\lambda_1,\dots,\lambda_t\in{\mathbb C}\cup\infty$.

{\rm(b)}
Each pair \eqref{a1def} with distinct eigenvalues $\lambda_1,\dots,\lambda_t\in{\mathbb C}\cup\infty$
is equivalent  to a pair of the form
\begin{equation}\label{min.case2}
\arraycolsep=0.24em
\left(
 \left[\begin{array}{ccc|cccc}
1&&&&&&\alpha _1\\
0&\ddots&&&&&\alpha _2 \\[-3pt]
&\ddots&1&&
\text{\raisebox{7pt}[0pt][0pt]{$0$}}
&&\vdots\\
&&0&&&&\alpha _m\\\hline
&&&1&\,0\!\!\!\\
&0&&&\ddots&\ddots\\
&&&&&1\!\!&\,0\!\!
 \end{array}
              \right],\,
\left[\begin{array}{ccc|cccc}
0&&&\beta_1&\beta_2&\dots &\beta _n\\
1&\ddots&&&&&\\[-3pt]
&\ddots&0& &&
\!\!\!\!\!\!0&\\
&&1&&&&\\\hline
&&&0&1\\
&0&&&\ddots&\ddots\\
&&&&&0\!\!&1\!\!
 \end{array}
 \right]
 \right)
\end{equation}
(which is obtained by an arbitrarily small perturbation of \eqref{(a1)}),
in which
\begin{equation}\label{i9i}
(-\beta _1,\dots,-\beta _n,\alpha _1,\dots,\alpha _m):=
\varepsilon (c_0,\dots,c_{r-1},1,0,\dots,0),
\end{equation}
$\varepsilon$ is an arbitrary nonzero complex number, and $c_{0},\dots,c_{r-1}$ are defined by
\begin{equation*}\label{amz}
c_0+c_1x+\dots+c_{r-1}x^{r-1}+x^r
:=\prod_{\lambda_i\ne\infty}
(x-\lambda _i)^{r_i}.
\end{equation*}
\end{theorem}

\begin{proof}
Let
$
(C,D)= \mc P\matp{\alpha _1\dots\alpha _m\\\beta _1\,\dots\,\beta _n}
$
denote the pair \eqref{min.case2}. Then
\begin{equation}\label{mmi}
\begin{split}
\setlength{\arraycolsep}{2pt}
(D^T,C^T)&=\left(\arr{{c|c}R_m&0\\\hline \ma{\ma{\beta _1\\\vdots\\\beta _n}&0&}&\ R^T_n},
\arr{{c|c}L_m&0\\\hline \ma{\\[-2pt]0\\\ma{\alpha _1&\dots&\alpha_m}}&\ L^T_n}\right)
                                       \\
     &\sim
\left(\arr{{c|c}R^T_n&
\ma{\ma{\beta _1\\\vdots\\\beta _n}&0&}
\\\hline 0&\ R_m\rule{0pt}{13pt}},
\arr{{c|c}L^T_n&\ma{\\[-2pt]0\\\ma{\alpha _1&\dots&\alpha_m}\!\!}\\\hline 0&\ L_m\rule{0pt}{13pt}}\right)\hspace{60pt}
                                \\
&\sim
\left(\arr{{c|c}L^T_n&
\ma{&0&\ma{\beta _n\\\vdots\\\beta _1}}
\\\hline 0&\ L_m\rule{0pt}{13pt}},
\arr{{c|c}R^T_n&\ma{\ma{\alpha _m&\dots&\alpha_1\!\!}\\[10pt] 0\\[-10pt]\phantom{.}}\\\hline 0&\ R_m\rule{0pt}{13pt}}\right)=\mc P\matp{\beta _n\,\dots\,\beta _1\\\alpha _m\dots\alpha _1},
\end{split}
\end{equation}
in which
the third pair is obtained from the second by reversing the order of rows in each horizontal strip and reversing the order of columns in each vertical strip.

By Theorem \ref{first}, there is a neighborhood of \eqref{(a1)}, in which each pair is equivalent to the pair
\begin{equation}\label{dlk}
\mc P\matp{\alpha _1\dots\alpha _m\\\beta _1\,\dots\,\beta _n}
\end{equation}
for some $\alpha _1,\dots,\alpha _m,\beta _1,\dots,\beta _n$.
The following three cases are possible.
\medskip

\emph{Case 1: $\alpha _m\ne 0$ in \eqref{dlk}.}
In this case,
$\mc P\matp{\alpha _1\dots\alpha _m\\\beta _1\,\dots\,\beta _n}\sim(I_{m+n-1},\Phi)$ with
\begin{align}\label{gtm}
 \Phi&:=\begin{bmatrix}
 -c_{m+n-2}&\dots&-c_1&-c_0\\
 1&&0&0\\&\ddots&&\vdots\\0&&1&0
 \end{bmatrix},\\ \label{gjm}
(c_0\dots,c_{m+n-2})&:=
\alpha _m^{-1}(
-\beta_1,\dots,-\beta _n,
\alpha _{1},\dots, \alpha_{m-1})
\end{align}
because
\begin{equation}\label{ode}
\mc P\matp{\alpha _1\dots\alpha _m\\\beta _1\,\dots\,\beta _n}(Q_{m-1}\oplus Z_{n})=
 (Q_{m}\oplus Z_{n-1}) (I_{m+n-1},\Phi),
\end{equation}
in which
\[
Q_p:=\mat{\alpha _m&\alpha _{m-1}&\alpha _{m-2}&\ddots\\
&\alpha _m&\alpha _{m-1}&\ddots\\
&&\alpha _m&\ddots\\
&&&\ddots\\
}\ (p\text{-by-}p),\qquad
Z_p:=\mat{0&&1\\&
\udots&\\1&&0}\ (p\text{-by-}p).
\]
For example, if $m=n=4$, then \eqref{ode} takes the form
\begin{multline*}
\left(\rule{0pt}{60pt}
\left[\begin{MAT}(e)[0pt]{ccc3cccc}
1&0&0&&&&\alpha_1\\
0&1&0&&&&\alpha_2\\
0&0&1&&&&\alpha_3\\
0&0&0&&&&\alpha_4
\\3
&&&1&0&0&0\\
&&&0&1&0&0\\
&&&0&0&1&0
\\
 \end{MAT}\right],
\left[\begin{MAT}(e)[0.5pt]{ccc3cccc}
0&0&0 & \beta_1&\beta_2&\beta_3&\beta_4\\
1&0&0&&&&\\
0&1&0&&&&\\
0&0&1&&&&
\\3
&&&0&1&0&0\\
&&&0&0&1&0\\
&&&0&0&0&1
\\
 \end{MAT}\right]
\right)\left[\begin{MAT}(e)[0pt]{ccc3cccc}
\alpha_4 &\alpha_3&\alpha_2 &&&&\\
&\alpha_4&\alpha_3&&&&\\
&&\alpha_4&&&&\\3
&&&&&&1
\\
&&&&&1&\\
&&&&1 &&\\
&&&1 &&&
\\
 \end{MAT}\right]
    \\ 
=\left[\begin{MAT}(e)[0.5pt]{cccc3ccc}
\alpha_4&\alpha_3&\alpha_2&\alpha_1&&&\\
&\alpha_4&\alpha_3&\alpha_2&&&\\
&&\alpha_4&\alpha_3&&&\\
&&&\alpha_4&&&
\\3
&&&&&&1 \\
&&&&&1 &\\
&&&&1 &&
\\
 \end{MAT}\right]
\left(\rule{0pt}{65pt}
I_7
,
\left[\begin{MAT}(b){ccc3cccc}
-c_6&-c_5&-c_4 & -c_3&-c_2&-c_1&-c_0\\
1&0&0&&&&\rule{0pt}{12pt}\\
0&1&0&&&&\\
0&0&1&&&&
\\3
&&&1&0&0&0\\
&&&0&1&0&0\\
&&&0&0&1&0
\\
 \end{MAT}\right]
\right)
\end{multline*}
in which
\[
\alpha_4 (c_0,c_1,c_2,c_3,c_4,c_5,c_6)=
(-\beta_1,-\beta_2,-\beta_3,-\beta_4,
\alpha_1,\alpha_2,\alpha_3).
\]

The Jordan canonical form of
\eqref{gtm} is  $J_{r_1} (\lambda_1)\oplus\dots\oplus J_{r_t}(\lambda_t)$ with distinct $\lambda_1,\dots,\lambda_t\in\mathbb C$; its characteristic polynomial is
\begin{equation}\label{uvr}
\begin{gathered}
(x-\lambda _1)^{r_1}\cdots(x-\lambda _t)^{r_t}=c_0+c_1x+\dots+c_{m+n-2}x^{m+n-2}+
x^{m+n-1}\\
=
\alpha_m^{-1}(
-\beta_1-\beta_2x-\dots-\beta _nx^{n-1}+
\alpha_1x^n+\alpha_2x^{n+1}+\dots+
\alpha_mx^{m+n-1}).
\end{gathered}
\end{equation}

We have proved that
\begin{equation}\label{mmw}
\mc P\matp{\alpha _1\dots\alpha _m\\\beta _1\,\dots\,\beta _n}\sim (I,\Phi)\sim \mc D_{r_1} (\lambda_1)\oplus\dots\oplus \mc D_{r_t}(\lambda_t
)\qquad \text{if $\alpha_m\ne 0$};
\end{equation}
it is a pair of the form \eqref{a1def}, which
proves the statement (a) in Case 1.

By \eqref{mmw}, each pair \eqref{a1def} with distinct \emph{nonzero} eigenvalues $\lambda_1,\dots,\lambda_t\in{\mathbb C}\cup\infty$ is equivalent to $\mc P\matp{\alpha _1\dots\alpha _m\\\beta _1\,\dots\,\beta _n}$ defined by \eqref{uvr}. Then \eqref{gjm} holds, and so $\mc P\matp{\alpha _1\dots\alpha _m\\\beta _1\,\dots\,\beta _n}$ is the pair \eqref{min.case2} defined by \eqref{i9i} with $\varepsilon =\alpha _m$. The pair $\mc P\matp{\alpha _1\dots\alpha _m\\\beta _1\,\dots\,\beta _n}$ is also equivalent to the pair \eqref{min.case2} defined by \eqref{i9i} with an arbitrary nonzero $\varepsilon$ since
\begin{equation}\label{3g3}
\setlength{\arraycolsep}{1.pt}
\left(\mat{L_m^T&P\\0&L_n},
\mat{R_m^T&Q\\0&R_n}
\right)\mat{I_{m-1}&0\\0&\delta I_n}
=
\mat{I_{m}&0\\0&\delta I_{n-1}}
\left(\mat{L_m^T&\,\delta P\\0&L_n},
\mat{R_m^T&\,\delta Q\\0&R_n}
\right)
\end{equation}
for an arbitrary nonzero $\delta$. This proves the statement (b) if all $\lambda _i\ne\infty$.
\medskip

\emph{Case 2: $\alpha _k\ne 0=\alpha _{k+1}=\dots=\alpha_m$ for some $k<m$ in \eqref{dlk}.}
Let us show that
\begin{equation}\label{v5t}
\mc P\matp{\alpha _1\dots\alpha _m\\\beta _1\,\dots\,\beta _n}=\mc P\matp{\alpha _1\dots\alpha _k\,0\dots 0\\\beta _1\ \dots\ \beta _n}\sim\mc P\matp{\alpha _1\,\dots\,\alpha _{k}\\\beta _1\,\dots\, \beta _n}
\oplus (J_{m-k}(0),I_{m-k})\qquad \text{ if }\alpha _k\ne 0.
\end{equation}
For clarity, we first prove the following special case of \eqref{v5t}:
\begin{equation}\label{v3t}
\mc P\matp{\alpha _1\,\alpha _2\:0\; 0\\ \beta _1\beta_2\beta_3\beta_4}\sim\mc P\matp{\alpha _1\, \alpha _2\\\beta _1\beta _2\beta_3\beta _4}\oplus(J_2(0),I_2)\qquad \text{if }\alpha _2\ne 0.
\end{equation}
The first pair in \eqref{v3t} is
\[
\setlength{\arraycolsep}{3pt}
(C,D):=\left(\arrr{
1&0&0 & &&&\alpha _1\\
0&1&0 & &&&\alpha _2\\
0&0&1 & &&&0\\
0&0&0 & &&&0\\\hline
&& & 1&0&0&0\\
&& & 0&1&0&0\\
&& & 0&0&1&0
},\:
\arrr{0&0&0 & \beta _1&\beta _2&\beta_3&\beta _4\\
1&0&0 & &&&\\
0&1&0 & &&&\\
0&0&1 & &&&\\\hline
&& & 0&1&0&0\\
&& & 0&0&1&0\\
&& & 0&0&0&1
}\right),\ \alpha _2\ne 0.
\]
It is sufficient to make zero the entry $(2,2)$ of $C$; i.e., to prove that
\begin{equation}\label{anb}
(C,D)\sim
\left(\,\left[\begin{MAT}(b)[1.5pt]{ccc3cccc}
1&0&0&&&&\alpha _1\\
0&0&0&&&&\alpha _2\\
0&0&1&&&&0\\
0&0&0&&&&0
\\3
&&&1&0&0&0\\
&&&0&1&0&0\\
&&&0&0&1&0
\addpath{(1,3,4)uurrddll}
\\
 \end{MAT}\right]\,,\,
\left[\begin{MAT}(b)[1.5pt]{ccc3cccc}
0&0&0 & \beta _1&\beta _2&\beta_3&\beta _4\\
1&0&0&&&&\\
0&1&0&&&&\\
0&0&1&&&&
\\3
&&&0&1&0&0\\
&&&0&0&1&0\\
&&&0&0&0&1
\addpath{(1,3,4)uurrddll}
\\
 \end{MAT}\right]
\right)
\end{equation}
since the pair $\left(\matt{0&1\\0&0}, \matt{1&0\\0&1}\right)$ in the squares is a direct summand.
We make this zero preserving the other entries by the following sequence of elementary transformations with $(C,D)$:
\begin{itemize}
  \item
Substituting column 7 multiplied by $\alpha _2^{-1}$ from column 2, we make zero the entry $(2,2)$ of $C$:
\[
\left(\arrr{1&*&0 & &&&\alpha _1\\
0&0&0 & &&&\alpha _2\\
0&0&1 & &&&0\\
0&0&0 & &&&0\\\hline
&& & 1&0&0&0\\
&& & 0&1&0&0\\
&& & 0&0&1&0
},\:
\arrr{0&* &0 & \beta _1&\beta _2&\beta_3&\beta _4\\
1&0&0 & &&&\\
0&1&0 & &&&\\
0&0&1 & &&&\\\hline
&& & 0&1&0&0\\
&& & 0&0&1&0\\
&*& & 0&0&0&1
}\right).
\]
This transformation may spoil the entries  denoted by $*$ in columns 2 of $C$ and $D$; we restore them as follow.

  \item
We restore column 2 of $C$ by adding column 1 (multiplied by a scalar) to column 2. This transformation spoils entry $(2,2)$ of $D$; we restore it and the entries denoted by stars in column 2 of $D$ by adding row 3 to rows 1, 2, and 7. We  obtain
 \[
\setlength{\arraycolsep}{4pt}
\left(\arrr{
1&0&* & &&&\alpha _1\\
0&0&* & &&&\alpha _2\\
0&0&1 & &&&0\\
0&0&0 & &&&0\\\hline
&& & 1&0&0&0\\
&& & 0&1&0&0\\
&& *& 0&0&1&0
},\:
\arrr{
0&0&0 & \beta _1&\beta _2&\beta_3&\beta _4\\
1&0&0 & &&&\\
0&1&0 & &&&\\
0&0&1 & &&&\\\hline
&& & 0&1&0&0\\
&& & 0&0&1&0\\
&& & 0&0&0&1
}\right).
\]

  \item
We restore column 3 of $C$ by adding columns 1, 6, and 7, which spoils column 3 of $D$. We restore it by adding row 4 and obtain \eqref{anb}, which proves \eqref{v3t}.
\end{itemize}

The equivalence \eqref{v5t} for an arbitrary pair $(C,D)=\mc P\matp{\alpha _1\dots\alpha _m\\\beta _1\,\dots\,\beta _n}$ with $\alpha _k\ne 0=\alpha _{k+1}=\dots=\alpha_m$ is proved in the same way: we make zero the entry $(k,k)$ of $C$ by adding the last column, which may spoil the entries $(1,k),\dots,(k-1,k)$ of  $C$; they are made zero by adding  columns $1,\dots,k-1$. This spoils
column $k$ of $D$; we restore it by row transformations. This spoils
column $k+1$ of $C$; we restore it by column transformations, and so on, until we obtain the equivalence \eqref{v5t}.

By \eqref{mmw} and \eqref{v5t},
\begin{equation}\label{doc}
\mc P\matp{\alpha _1\dots\alpha _m\\\beta _1\,\dots\,\beta _n}=\mc P\matp{\alpha _1\dots\alpha _k\,0\dots 0\\\beta _1\ \dots\ \beta _n}\sim
\mc D_{r_1} (\lambda_1)\oplus\dots\oplus \mc D_{r_{t-1}}(\lambda_{t-1})
\oplus \mc D_{m-k}(\infty),
\end{equation}
which proves the statement (a) in Case 2. Since
\begin{equation}\label{uvr1}
\alpha_k^{-1}(-\beta_1
 -\dots-\beta _nx^{n-1}+\alpha_1x^{n}
+\dots+\alpha_kx^{n+k-1})
=
\prod_{i=1}^{t-1}(x-\lambda _i)^{r_i},
\end{equation}
the statement (b) holds for $\varepsilon =\alpha_k$.
It holds for an arbitrary nonzero $\varepsilon$  due to \eqref{3g3}.
\medskip

\emph{Case 3: $\alpha _1=\dots=\alpha _m=0$ in \eqref{dlk}}; that is, $(C,D)=\mc P\matp{0\  \dots\  0\ \\\beta _1\,\dots\,\beta _n}$. If $\beta _1=\dots=\beta _n=0$, then $(C,D)$ is the pair \eqref{(a1)}.
Let $\beta _k\ne 0=\beta _{k-1}=\dots=\beta _1$ for some $k\ge 1$. By \eqref{mmi}, \eqref{doc}, and \eqref{uvr1}, we have
\[
(D^T,C^T)\sim \mc P\matp{\beta _n\,\dots\,\beta _1\\0\ \dots\ 0}
=\mc P\matp{\beta _n\,\dots\,\beta _k\,0\dots 0\\ 0\ \dots\ 0}\sim
\mc D_{r_1} (\mu_1)\oplus\dots\oplus\mc D_{r_t} (\mu_t)
\oplus \mc D_{k-1}(\infty),
\]
in which $\mu_1,\dots,\mu_{t}$ are distinct and
\[
(x-\mu _1)^{r_1}\cdots(x-\mu _{t})^{r_{t}}=\beta _k^{-1}
(\beta_nx^m+\beta_{n-1}x^{m+1}
+\dots+\beta _kx^{m+n-k}).
\]

Let $\beta _n=\beta_{n-1}=\dots=\beta _{l+1}= 0\ne\beta _{l}$ for some $l\ge k$. Then
\[
(x-\mu _1)^{r_1}\cdots(x-\mu _{t})^{r_{t}}=\beta _k^{-1}
(\beta_lx^{m+n-l}+\beta_{l-1}x^{m+n-l+1}
+\dots+\beta _kx^{m+n-k}).
\]
Set $\mu _t=0$ and rewrite this equality as follows:
\begin{equation}\label{214}
(x-\mu _1)^{r_1}\cdots(x-\mu _{t-1})^{r_{t-1}}x^{m+n-l}=\beta _k^{-1}
(\beta_l+\beta_{l-1}x
+\dots+\beta _kx^{l-k})x^{m+n-l}.
\end{equation}
This proves that
\[
(D^T,C^T)\sim
\mc D_{r_1} (\mu_1)\oplus\dots\oplus\mc D_{r_{t-1}} (\mu_{t-1})\oplus
\mc D_{m+n-l}(0)
\oplus \mc D_{k-1}(\infty),
\]

If $k=l$, then $t=1$ and
$
(C,D)\sim \mc D_{m+n-l}(\infty)
\oplus \mc D_{k-1}(0).
$

Let $l>k$, then $t\ge 2$.
Setting
\begin{equation*}\label{7yt}
\lambda _1:=\mu_1^{-1},\ \dots,\ \lambda _{t-1}:=\mu_{t-1}^{-1},
\end{equation*}
 we find that
\[
(C,D)\sim \mc D_{r_1} (\lambda_1)\oplus\dots\oplus \mc D_{r_{t-1}}(\lambda_{t-1})
     \oplus
\mc D_{m+n-l}(\infty)
\oplus \mc D_{k-1}(0).
\]
This proves the statement (a) in Case 3.

Replacing $x$ by $x^{-1}$ in the polynomials \eqref{214} and equating the leading coefficients, we obtain
\[
(x^{-1}-\lambda_1^{-1})^{r_1}
\cdots(x^{-1}-\lambda_{t-1}^{-1})^{r_{t-1}}=\beta _k^{-1}
(\beta_l+\beta_{l-1}x^{-1}
+\dots+\beta _kx^{k-l}),
\]
and so
\[
(x-\lambda_1)^{r_1}\cdots
 (x-\lambda_{t-1})^{r_{t-1}}
 =
 \beta _l^{-1}
(\beta_k+\beta_{k+1}x
+\dots+\beta _lx^{k-l}).
\]
This proves the statement (b) for $\varepsilon =-\beta _l$. It holds for an arbitrary nonzero $\varepsilon $ due to \eqref{3g3}.
\end{proof}

\section{Perturbations of Jordan matrices}
\label{jord}

By Lipschitz's property (see \cite{rod} or \cite{ala}), each matrix that is obtained by an
arbitrarily small
perturbation of $I_n$ is reduced to
$I_n$ by equivalence transformations that are close to the identity transformation. Hence, each pair that is obtained by an
arbitrarily small
perturbation of
$(I_n,B)$ is reduced to a pair of the form
$(I_n,C)$ by equivalence transformations that are close to the identity transformation.

Hence, the theory of perturbations of matrix pairs $(A,B)$ with a nonsingular $A$ under equivalence is reduced to the theory of perturbations of square matrices under similarity. By Theorem
\ref{th}, it reduces to the theory of perturbations of Jordan matrices with a single eigenvalue.

The closures of orbits of Jordan matrices under similarity have been described by
Den  Boer and Thijsse \cite{den} and, independently, by Markus and Parilis \cite{mar}; see also \cite[Theorem 2.1]{kag2}. In this section, we describe the closures of orbits of Jordan matrices in the form that is used in the proof of Theorem I.

\begin{theorem}\label{jut}
Let $J$ be a
Jordan matrix with a
single eigenvalue
$\lambda $.
\begin{itemize}
  \item[\rm(a)] If $J$ is a Jordan block, then $\langle
J\rangle$ has no successors.

  \item[\rm(b)]
Let $J$ have at least $2$ Jordan blocks. Write it as follows:
\begin{equation}\label{aa1}
J=P\oplus J_{p}(\lambda)\oplus
J_{q}(\lambda)\oplus Q,\qquad
p\le q,
\end{equation}
in which $P$ is a direct sum of Jordan blocks of sizes $\le p$ and $Q$ is a direct sum of Jordan blocks of sizes $\ge q$ ($P$ and $Q$ can be zero).
Define the Jordan matrix
\begin{equation}\label{aa2}
J_{p,q}:=P\oplus J_{p-1}(\lambda)\oplus
J_{q+1}(\lambda)\oplus Q,
\end{equation}
in which $J_{p-1}(\lambda)$ is
absent if $p=1$.
Then
$\langle
J_{p,q}\rangle$ immediately succeeds $\langle J\rangle$, and
each immediate
successor of $\langle J\rangle$ is $\langle
J_{p,q}\rangle$ for some $p$ and $q$.
\end{itemize}
\end{theorem}

The \emph{Weyr
characteristic} of a
square matrix $A$ for an eigenvalue
$\lambda$ is the non-increasing
sequence $(m_{1},
m_{2},\dots)$ in which
$m_i$ is the number of
Jordan blocks
$J_l(\lambda )$ of
size $l\ge i$ in the
Jordan form of $A$.

In the proof of Theorem \ref{jut}, we use the fact that
each nilpotent matrix $A$ is similar to a matrix of the form
\begin{equation}\label{weyr}
W=\begin{bmatrix}
0_{m_{1}}&F_1&&0\\
&0_{m_{2}}&\ddots&\\
&&\ddots&F_{k-1}\\
0&&&0_{m_{k}}
\end{bmatrix},\qquad F_i:= \begin{bmatrix}
I_{m_{i+1}}\\0
\end{bmatrix},
\end{equation}
which
is permutation similar to the Jordan canonical form of $A$.
The matrix $W$
has been called in \cite{ser} the \emph{Weyr
canonical form} of $A$. Now this term is generally accepted; see historical remarks in \cite[pp.
80--82]{omear}. The
Weyr characteristic of
$A$ for its single eigenvalue $0$ is  $(m_{1},
m_{2},\dots)$. The latter holds since the equalities
\begin{equation*}\label{nilmah}
W^2=
\begin{bmatrix}
0_{m_{1}}&0&F_1F_2&&0\\
&0_{m_{2}}&0&\ddots\\
&&0_{m_{3}}&\ddots&F_{k-2}F_{k-1}\\
&&&\ddots&0
\\ 0&&&&0_{m'_{k}}
\end{bmatrix},\ \ W^3=\cdots,\ \ \dots
\end{equation*}
imply
\begin{align*}\label{tre}
m_1&=\nullity W=\nullity
A,\\ m_1+m_2&=\nullity
W^2=\nullity A^2,\\
m_1+m_2+m_3&=\nullity
W^3=\nullity
A^3,\\&\dots
\end{align*}

\begin{lemma}      \label{caose4}
Let $J$ and
$J'$ be
Jordan matrices with a
single eigenvalue
$\lambda $. Let
$(m_{1}, m_{2},\dots)$
and $(m'_{1},
m'_{2},\dots)$ be
their Weyr
characteristics. Write
\begin{equation}\label{jyr}
s_i:=m_1+\dots+m_i,\qquad
s'_i:=m'_1+\dots+m'_i
\end{equation}
for $i=1,2,\dots$.
Then
\begin{equation}\label{kli}
\langle J\rangle\le
\langle J'\rangle
\quad\Longleftrightarrow\quad
s_i\ge s'_i\text{ for
all }i.
\end{equation}
\end{lemma}

\begin{example}
Let
\[
J=J_3(\lambda )\oplus
J_4(\lambda )\oplus J_4(\lambda ),\qquad
J'=J_3(\lambda )\oplus
J_3(\lambda )\oplus J_5(\lambda ).
\]
Then
\[
\begin{array}{llll}
m_1=m_2=m_3=3,\quad&
m_4=2,\quad&
m_5=0,\quad&m_6=m_7=\dots=0,
                   \\
  m'_1=m'_2=m'_3=3,\quad&
m'_4=1,\quad& m'_5=1,\quad&
m'_6=m'_7=\dots=0,
\end{array}
\]
and so
\[
\begin{array}{lllll}
s_1=3,\quad& s_2=6,\quad&
s_3=9,\quad&
s_4=11,\quad& s_5=s_6=\dots=11,
               \\
s'_1=3,\quad& s'_2=6,\quad&
s'_3=9,\quad& s'_4=10,\quad&
s'_5=s'_6=\dots=11.
\end{array}
\]
Hence, $\langle J\rangle<
\langle J'\rangle$.
\end{example}

\begin{proof}[Proof of Lemma \ref{caose4}] Let $J$ be a Jordan matrix with a single eigenvalue $\lambda $. Then $/\ll J-\lambda I\rr= \ll J\rr-\lambda I$ and   $\overline{\ll J-\lambda I\rr}= \overline{\ll J\rr}-\lambda I$ for their closures. Hence, we must prove  \eqref{kli} only for $J$ and $J'$  with the single eigenvalue $\lambda =0$.

$\Longleftarrow$.
Let $W$ and $W'$ be Weyr
canonical matrices of the same size with the single eigenvalue $0$. Let their
Weyr characteristics
$(m_1, m_2, \dots,
m_k)$ and $(m'_1, m'_2, \dots,
\ m'_k)$ satisfy
$s_1\ge s'_1$, $s_2\ge
s'_2$, \dots .
Then for each
sufficiently small
$\varepsilon $ the
Weyr canonical form of
$\varepsilon W'+W$ is
$W'$. If $\varepsilon_i \to 0$, then $\varepsilon_i W'+W \to
W$. Hence $\langle W\rangle\le
\langle W'\rangle$.
\medskip

$\Longrightarrow$. Let $J$ be a Jordan matrix  with the single eigenvalue $\lambda =0$.
Let $J'$ be a Jordan
matrix such that each
neighborhood of $J$
contains a matrix whose Jordan canonical form is $J'$. This
means that there is a
convergent sequence
\begin{equation}
\label{iyt} A_1, \,
A_2, \, \dots \, \to
\, J
\end{equation}
in which all $A_i$ are
similar to $J'$. All $A_i$  have the
same characteristic
polynomial $f(x)$.
Since the coefficients
of characteristic
polynomial
continuously depend on
the matrix entries,
$f(x)$ is also the
characteristic
polynomial of
$J$. Hence, $f(x)=x^n$, and so $J'$ is nilpotent.

Since all $A_i$ are
similar to $J'$, they
have the same Weyr
canonical form
\begin{equation*}\label{weyr1}
S^{-1}_i A_i
S_i=\begin{bmatrix}
0_{m'_{1}}&F_1'&&0\\
&0_{m'_{2}}&\ddots&\\
&&\ddots&F_{k-1}'\\
0&&&0_{m'_{k}}
\end{bmatrix},\qquad F_i':= \begin{bmatrix}
I_{m'_{i+1}}\\0
\end{bmatrix},
\end{equation*}
in which $(m'_{1},
m'_{2},\dots)$ is the
Weyr characteristic of
$J'$. Applying the
Gram--Schmidt
orthogonalisation
process to the columns
of $S_i$, we obtain a
unitary matrix
$U_i=S_iR_i$, where
$R_i$ is a nonsingular
upper-triangular
matrix. Then
\begin{equation*}\label{nilmax}
U^{-1}_i A_i
U_i=R^{-1}_i\cdot S^{-1}_i A_i
S_i\cdot R_i=
 \begin{bmatrix}
 0_{m'_{1}}&V^{(i)}_1&*&\dots&*\\
 &0_{m'_{2}}&V^{(i)}_2&\ddots&\vdots\\
 &&0_{m'_{3}}&\ddots&*\\
 &&&\ddots&V^{(i)}_{k-1}\\
 0&&&&0_{m'_{k}}
 \end{bmatrix},
\end{equation*}
in which every $V^{(i)}_j$
is an $m'_i \times
m'_{i+1}$ matrix with
linearly independent
columns.

The set of matrices $U_1,U_2,\dots$ is bounded since
each entry of a unitary matrix has modulus $\le 1$. Hence this set has a limit
point, which we denote
by $U$. Deleting some $A_i$
in \eqref{iyt} if
necessarily, we make $U_i \to
U$.  Since each $U_i$ is unitary, we have $U_iU^*_i=I$, and so $UU^*=I$. Hence $U$ is unitary and
\begin{equation*}\label{bweyr}
U^{-1}_i A_i U_i\to
U^{-1} JU=
 \begin{bmatrix}
  0_{m'_{1}}&V_1&*&\dots&*\\
 & 0_{m'_{2}}&V_2&\ddots&\vdots\\
 && 0_{m'_{3}}&\ddots&*\\
 &&&\ddots&V_{k-1}\\
 0&&&& 0_{m'_{k}}
 \end{bmatrix},
\end{equation*}
in which $V^{(i)}_1
\to V_1$, \dots,
$V^{(i)}_{k-1} \to
V_{k-1}$. Note that the columns of some $V_i$ can be linearly dependent.

Therefore,
\[
m_1=\nullity J=
\nullity U^{-1}
JU\ge  m_1'.
\]
Since
\[
U^{-1}
J^2U=\begin{bmatrix}
0_{m'_{1}}&0&V_1V_2&&0\\
&0_{m'_{2}}&0&\ddots\\
&&0_{m'_{3}}&\ddots&V_{k-2}V_{k-1}\\
&&&\ddots&0
\\ 0&&&&0_{m'_{k}}
\end{bmatrix},
\]
we have
\[
 m_1+m_2= \nullity J^2= \nullity U^{-1} J^2U \ge m_1'+m_2',
\]
and so on, which proves ``$\Longrightarrow$'' in
\eqref{kli}.
\end{proof}

\begin{proof}[Proof of Theorem \ref{jut}]
(a) Let $J=J_p(\lambda)$ and  $\langle J\rangle\le
\langle J'\rangle$.
By \eqref{kli}, $m'_1\le m_1=1$. Since $m'_1$ is the number of Jordan blocks, $J'$ is a Jordan block. Since $J$ and $J'$ have the same size, $J'=J_p(\lambda )=J$.

(b) Denote
by $(m_{1}(X),
m_{2}(X),\dots)$ the
Weyr characteristic of a matrix
$X$ and write
$s_i(X):=m_1(X)+\dots+m_i(X).$
Let $A$, $B$, and $C$ be square matrices with a single eigenvalue. Since $m_i(A\oplus
B)=m_i(A)+m_i(B)$, we have $s_i(A\oplus
B)=s_i(A)+s_i(B)$.
Thus, $s_i(A\oplus
B)\le s_i(A\oplus C)$
if and only if
$s_i(B)\le s_i(C)$. By
\eqref{kli},
\begin{equation}\label{kut}
\langle A\oplus B\rangle\le
\langle A\oplus C\rangle \quad\Longleftrightarrow\quad
\langle B\rangle\le \langle C\rangle.
\end{equation}

Let $(m_1, m_2, \dots)$ and $(\wt m_1, \wt m_2, \dots)$ be the Weyr characteristics of the matrices \eqref{aa1} and \eqref{aa2}.
Then $
\wt m_p=m_p-1$,
$\wt m_{q+1}=m_{q+1}+1$, the other $\wt m_i=m_i$,
and so
\begin{equation}\label{hsd}
\wt s_p=s_p-1,\
\wt s_{p+1}=s_{p+1}-1,\
\dots,\
\wt s_q=s_q-1,\quad\text{the
other }\wt s_i=s_i
\end{equation}
in the notation \eqref{jyr}. Let us prove the following three facts.

\smallskip

\emph{Fact 1:
$\ll J\rr<\ll J_{p,q}\rr.$}
It follows from \eqref{kut} and the inequality $\ll J_p(\lambda)\oplus
J_q(\lambda)\rr<\ll J_{p-1}(\lambda)\oplus
J_{q+1}(\lambda)\rr$, which holds by  \eqref{kli} and \eqref{hsd}.

\smallskip

\emph{Fact 2: if
$J'$ is a Jordan
matrix with the single
eigenvalue $\lambda $,
then}
\begin{equation}\label{yte}
  \ll J\rr < \ll J'\rr
\quad\Longrightarrow\quad
\ll J\rr <\ll J_{p,q}\rr  \le \ll J'\rr
\textit{ for some } p,q.
\end{equation}
Due to \eqref{kut}, it is sufficient to prove \eqref{yte} for
$J$ and
$J'$ that have no
common Jordan blocks. By the assumptions of Theorem \ref{jut}(b), $J$ has at least two Jordan blocks.
Let $p$ and $q$ be such that
\begin{equation*}\label{fde}
J=J_p(\lambda )\oplus
J_q(\lambda )\oplus Q,\qquad
p\le q,
\end{equation*}
in which all Jordan
blocks of $Q$ are of
size $\ge q$. Let us
prove that
\begin{equation*}\label{few}
J_{p,q}=J_{p-1}(\lambda)\oplus
J_{q+1}(\lambda)\oplus Q
\end{equation*}
satisfies \eqref{yte}.

By  $\ll J'\rr \ge
\ll J\rr$ and Lemma \ref{caose4}, $s'_i\le s_i$
for all $i$.
By Step 1,
$\ll J_{p,q}\rr>\ll J\rr$. We must prove that
$\ll J'\rr \ge
\ll J_{p,q}\rr$; i.e., that $s_i'\le\wt s_i$ for all $i$.  Due to
\eqref{hsd}, it
suffices to prove that
\begin{equation}\label{yre}
s_p'<s_p,\ \
s_{p+1}'<s_{p+1},\ \
\dots,\ \ s_q'<s_q.
\end{equation}

Since $J$ and
$J'$ do not have
common Jordan blocks, $J'$ does not contain $J_p(\lambda )$, and so
\begin{equation*}\label{teff}
\begin{array}{l}
  s_1=m_1=\dots=m_p>m_{p+1}\\[-6pt]
  {\text{\rotatebox{270}{$\ge$}}}\\[7pt]
  s_1'=m'_1\ge\cdots\ge m'_p=m'_{p+1}\\
\end{array}
\end{equation*}
Thus, $m_p\ge m_p'$.

If $m_p= m_p'$, then
\[
\begin{array}{l}
  s_1=m_1=\dots=m_p>m_{p+1}\\[-6pt]
  {\text{\rotatebox{270}{$=$}}}\\[7pt]
  s_1'=m'_1=\cdots= m'_p=m'_{p+1}\\
\end{array}
\]
Hence, $s_1=s_1'$, $s_2=s_2'$, \dots, $s_p=s_p'$, $s_{p+1}=s_p+m_{p+1}<s_p'+m_{p+1}'=s_{p+1}'$,
which contradicts $s_{p+1}\ge s_{p+1}'$.

Therefore, $m_p> m_p'$,
$s_p=s_{p-1}+m_p>s_{p-1}'+m_p'=s_p'$, and so $s_p>s_p'$, which proves \eqref{yre} if $p=q$.

Let $p<q$. Then $J$ has only one
$J_p(\lambda )$, which means that $m_p=m_{p+1}+1$. Since $m_p> m_p'$, we have $m_p-1\ge m_p'$, and so
\[
\begin{array}{l}
\hspace{-17pt}  m_p-1=m_{p+1}=m_{p+2}=\dots=m_q\\[-6pt]
  {\text{\rotatebox{270}{$\ge$}}}\\[7pt]
  m_p'\ =m'_{p+1}\ge m'_{p+2}\ge\cdots\ge m'_q\\
\end{array}
\]
We obtain consistently $s_p>s_p'$, $s_{p+1}=s_p+m_{p+1}>s_p'+m'_{p+1}= s_{p+1}'$, \dots, $s_q=s_{q-1}+m_q>s_{q-1}'+m_q'= s_q'$, which proves \eqref{yre} if $p<q$.

\smallskip

\emph{Fact 3: if $J'$ is a Jordan
matrix with the single
eigenvalue $\lambda $, then}
\begin{equation*}\label{ytr}
\ll J\rr < \ll
J'\rr\le
\ll J_{p,q}\rr
\quad\Longrightarrow\quad
J'=J_{p,q}
\end{equation*}
\emph{up to permutations of Jordan blocks in $J'$.}

On the contrary, let
$\ll J\rr<\ll
J'\rr<
\ll J_{p,q}\rr$ for some $J'$. By Fact 2, we can take $J'=J_{p',q'}$ for some $p'\le q'$.

Write $t(J):=(t_1,t_2,\dots)$, in which $t_i$ is the number of $i\times i$ Jordan blocks in $J$. Then $s(J):=t_1+t_2+\cdots$ is the number of Jordan blocks in $J$.

Let $u=(u_1,\dots,u_s)$ and $u=(v_1,\dots,v_s)$ be two sequences of nonnegative integers.
Define two lexicographical orders:
\begin{align*}
u\stackrel{l}{\prec}v\quad&\text{if }u_1=v_1,\ \dots, \ u_{k-1}=v_{k-1},\ u_k<v_k\text{ for some }k\ge 1;\\
u\stackrel{r}{\prec}v\quad&\text{if }u_k<v_k,\ u_{k+1}=v_{k+1},\ u_{k+2}=v_{k+2},\ \dots\text{ for some }k\ge 1.
\end{align*}

By Fact 2, the inequality $\ll
J_{p',q'}\rr<\ll J_{p,q}\rr$ implies that $J_{p,q}$ is obtained from $J_{p',q'}$ by a sequence of replacements of type $J\!\!\re\!\!\! J_{s,r}$:
\begin{equation}\label{x3w}
J_{p',q'}\re (J_{p',q'})_{r_1,s_1}\re
((J_{p',q'})_{r_1,s_1})_{r_2,s_2}\re\cdots\re
J_{p,q}.
\end{equation}
Therefore,
\begin{itemize}
  \item[(i)] $s(J_{p',q'})\ge s(J_{p,q})$,
  \item[(ii)]
  if $s(J_{p',q'})=s(J_{p,q})$, then $t(J_{p',q'})\stackrel{l}{\preceq} t(J_{p,q})$, and
  \item[(iii)] $t(J_{p',q'})\stackrel{r}{\succeq} t(J_{p,q})$
\end{itemize}
since the analogous statements hold for each of the replacements \eqref{x3w}.

Let $s(J_{p',q'})> s(J_{p,q})$. Then
$J=J_1(\lambda )\oplus\cdots$ and $p=1$.
Hence $q\le p'$, and so $t(J_{p',q'})\stackrel{r}{\prec} t(J_{p,q})$, which contradicts (iii).

Thus, $s(J_{p',q'})=s(J_{p,q})$. If $p'<p$, then (ii) does not hold. If $q'>q$, then (iii) does not hold. Hence, $p\le p'\le q'\le q$, which contradicts with $(p',q')\ne (p,q)$.
\end{proof}

The following theorem ensures Theorem III(v).

\begin{theorem}      \label{kur}
{\rm(a)}
All matrices in a sufficiently small neighborhood of
\begin{equation*}\label{moe}
J_m(\lambda)\oplus
J_n(\lambda),\qquad m\le n
\end{equation*}
are similar to matrices
of the form
\begin{equation}\label{cfq}
J_{m-r}(\lambda)
\oplus J_{n+r}(\lambda),\qquad 0\le r\le m.
\end{equation}

{\rm(b)} Each matrix \eqref{cfq} with $r>0$ is similar to
\begin{equation}\label{tgt}
\mat{J_m(\lambda )&\Delta_r(\varepsilon )^T\\0&J_n(\lambda )},
\end{equation}
in which $\Delta_r(\varepsilon )$ is defined in \eqref{rdpd}, and $\varepsilon$ is an arbitrary nonzero complex number.
\end{theorem}

\begin{proof}
(a) This statement follows from Theorem \ref{jut}(b).
\medskip

(b) We make $\varepsilon=1$ in \eqref{tgt} preserving the other entries by the following similarity transformation: we divide by $\varepsilon$ the first horizontal strip, then multiply by $\varepsilon$ the first vertical strip.  In the obtained matrix
\begin{equation}\label{12c}
\begin{MAT}(@)[1.5pt]{ccccccccccccccl}
 &&\vphantom{A_A}&\!\!\!\mathit{\scriptstyle r+1}\!\!\!\!\!\!& &&{\scriptstyle m}&\mathit{\scriptstyle 1}&&&\!\!\!\!\!\!{\scriptstyle m-r}\!\!\!&&&&\\ 
 \,\lambda &1&&& &&&\vdots&&&&&&&\mathit{\scriptstyle 1}\\ 
 &\ddotc&\ddotc &&&&&0&&&&&&&\\ 
  &&\lambda &1 &&&&1&&\vphantom{A_A}
  &&&&&{\,\scriptstyle r}\\ 
  &&&\lambda &1 &&&0&&&&&&&\mathit{\,\scriptstyle r+1}\!\!\!\!\!\!\!\!\\   
&&&&\lambda &\ddotc &&0&&&&&&&\\ 
&&&&&\ddotc&1&\vdots&&&&&&&\\  
 &&&&&&\lambda &0&&\vphantom{A_A}
 &&&&&{\,\scriptstyle m}\\  
 &&&0&{\ze}&&&\lambda &1&&&&&&\mathit{\scriptstyle 1}\\   
 &&&&0&\ddotc&&&\lambda &\ddotc&\vphantom{\ze}
&
&
&&\\  
&&&&&\ddotc&\ze&&&\ddotc&1&&&&\\  
 &&&&&&0&&\vphantom{A_A}
 &&\lambda &1&&&{\,\scriptstyle m-r}\!\!\!\!\!\!\!\!\\
 &&&&&&&&&&&\lambda &\ddotc&&\\
  &&&&&&&&&&&&\ddotc&1&\\
 &&&&&&&
 &
 &&&&&\lambda &{\,\scriptstyle n}
 \addpath{(0,0,4)rrrrrrrrrrrrrruuuuuuuuuuuuuu%
 lllllllllllllldddddddddddddd}
 \addpath{(3,3,4)rrrrrrrruuuullllllll}
 \addpath{(3,3,4)uuuuuuuurrrrdddddddd}
 \addpath{(0,7,2)rrrrrrrrrrrrrrr}
\addpath{(7,0,2)uuuuuuuuuuuuuuu}
\addpath{(3,11,.)uuu}
\addpath{(11,7,.)uuuuuuu}
\addpath{(7,11,.)rrrrrrr}
\addpath{(11,3,.)rrr}
 \\
 \end{MAT}
\end{equation}
we make zero the entry ``1'' to the left of $\varepsilon=1$ by the following similarity transformations (every $\ze$
denotes the zero entry that first is transformed to $-1$ and then is restored to $0$; compare with  \eqref{uuu}):
\begin{itemize}
  \item Make zero the entry ``1'' to the left of $\varepsilon=1$ by subtracting columns $1,2,\dots,m-r$ of the second vertical strip from columns $r+1, r+2,\dots, m$ of the first vertical strip, respectively. Thus, the marked $(m-r)\times(m-r)$ subblock in the $(2,2)$th block of the matrix \eqref{12c} is subtracted from the marked $(m-r)\times(m-r)$ subblock in the $(2,1)$th block.

  \item Make the inverse transformations of rows,      adding rows $r+1,\dots,m$  of the first horizontal strip to rows $1,\dots,m-r$ of the second horizontal strip. Thus, the $(m-r)\times(m-r)$ subblock in the $(1,1)$th block is added to the $(m-r)\times(m-r)$ subblock in the $(2,1)$th block, restoring it.
\end{itemize}
The $(m-r)\times(m-r)$ marked subblock in the $(1,1)$th block of the obtained matrix is a direct summand, and so the obtained matrix is permutation similar to \eqref{cfq}.
\end{proof}

\section{Theorem II follows from Theorem I}\label{sdsx}

Theorem II is formulated in terms of coin moves and proved sketchily by Edelman, Elmroth, and K\r{a}gstr\"{o}m \cite[Theorem 3.2]{kag2}.
In this section  (which can be read independently of Sections \ref{prel}--\ref{jord}), we give a detailed proof of Theorem II, deriving it from Theorem I.  It is sufficient to prove the following statement:
\begin{equation}\label{aam}
\parbox[c]{0.84\textwidth}{\it Let a Kronecker pair $\mc B$ be obtained from a Kronecker pair  $\mc A$ by a replacement {\rm(j)} from Theorem I, where $\text{\rm j}\in\{\text{\rm i},\text{\rm ii},\dots,\text{\rm vi}\}$. Then $\ll\mc B\rr$ immediately succeeds $\ll\mc A\rr$ if and only if\/ {\rm(j)} is the
replacement {\rm(j$'$)} from Theorem II.}
\end{equation}

Let us show that \eqref{aam} holds for the pair $\mc A$ given in \eqref{hwm}.
\medskip

\emph{Case 1:
{\rm (j)} is the replacement {\rm(i)}:}
\begin{equation}\label{eex}
\mc L_{n_i} \oplus
\mc L_{n_j} \re
\mc L_{n_i+1} \oplus
\mc L_{n_j-1}
,\quad \text{in which }n_i+2\le n_j.
\end{equation}

${\Longrightarrow}$.
Let $\ll\mc B\rr$ immediately succeed $\ll\mc A\rr$. We must prove that \eqref{eex} is the replacement (i$'$). To the contrary,
let $i+2\le  j$, $n_i<n_{i+1}<n_j$, and $n_i+3\le n_j$. If $n_i+2\le n_{i+1}$, then \eqref{eex} is the following composition of replacements of type (i):
\[\mc L_{n_i} \oplus \mc L_{n_{i+1}} \oplus
\mc L_{n_j}
               \re
\mc L_{n_i+1}
 \oplus \mc L_{n_{i+1}-1}
\oplus
\mc L_{n_j}
               \re
               \mc L_{n_i+1}
 \oplus \mc L_{n_{i+1}}
\oplus
\mc L_{{n_j}-1}.
 \]
By Theorem I,
\[\ll\mc L_{n_i} \oplus \mc L_{n_{i+1}} \oplus
\mc L_{n_j}\rr
               <
\ll\mc L_{n_i+1}
 \oplus \mc L_{n_{i+1}-1}
\oplus
\mc L_{n_j}\rr
               <
\ll\mc L_{n_i+1}
 \oplus \mc L_{n_{i+1}}
\oplus
\mc L_{{n_j}-1}\rr,
 \]
and so $\ll\mc B\rr$ is not an immediate successor of $\ll\mc A\rr$.
If $n_i+1= n_{i+1}$, then $n_{i+1}+2\le n_j$ and \eqref{eex} is the following composition of replacements of type (i):
\[\mc L_{n_i} \oplus \mc L_{n_{i+1}} \oplus
\mc L_{n_j}
               \re
\mc L_{n_i}
 \oplus \mc L_{n_{i+1}+1}
\oplus
\mc L_{n_j-1}
               \re
               \mc L_{n_i+1}
 \oplus \mc L_{n_{i+1}}
\oplus
\mc L_{{n_j}-1}.
 \]
Thus, $\ll\mc B\rr$ is not an immediate successor of $\ll\mc A\rr$ too.

${\Longleftarrow}$.
Let $\mc B$ be obtained from $\mc A$ by replacement (i$'$). Let $\mc B$ can be also obtained from $\mc A$ by a sequence
\begin{equation*}\label{vbv}
\mc A=\mc A_1 \stackrel{\varphi_1}{\longmapsto}
\mc A_2
\stackrel{\varphi_2}{\longmapsto}
\mc A_3
\stackrel{\varphi_3}{\longmapsto}
\cdots
\stackrel{\varphi_p}{\longmapsto}
\mc A_{p+1}=\mc B
\end{equation*}
of replacements of types (i)--(vi). In order to show that $\ll\mc B\rr$ is an immediate successor of $\ll\mc A\rr$, we must prove that $p=1$.

All replacements $\varphi _1,\dots,\varphi _p$ are not of
\begin{itemize}
  \item type (vi) since $\mc A$ and $\mc B$ have the same number $\underline s$ of summands $\mc L_i^T$, but (vi) decreases the number
      $\underline s$ and this number cannot be restored by (i)--(v);

  \item type (iv) since it increases the number $m_1+\dots+m_{\underline s}$ whereas this number is not changed by (i), (ii), (iii), and (v);
  \item type (iii) since it increases $n_1+\dots+n_{\tilde s}$;

\item type (v) with $\lambda =\lambda _i$ since it increases $\sum_{p<q}(k_{iq}-k_{ip})$ whereas this number is not changed by (i) and (ii);

\item type (ii) since it decreases $\sum_{i<j}(m_{j}-m_{i})$.
\end{itemize}
Therefore, all $\varphi _1,\dots,\varphi _p$ are replacements of type (i). Since each replacement (i$'$) is not the composition of several replacements of type (i), $p=1$, and so $\ll\mc B\rr$ is an immediate successor of $\ll\mc A\rr$.
We have proved \eqref{aam} in Case 1.
\medskip

\emph{Case 2:
{\rm (j)} is the replacement {\rm(ii)}.} The statement \eqref{aam} is proved in this case by transposing the matrices in Case 1.
\medskip

\emph{Case 3:
{\rm (j)} is the replacement {\rm(iii)}:}
\begin{equation}\label{eex3}
\mc L_{n}\oplus {\cal D}_k(\lambda_i) \re
\mc L_{n+1} \oplus \mc D_{k-1}(\lambda_i ),
\end{equation}
in which $(n,k)=(n_l,k_{ij})$ for some $l$, $i$, and $j$.

${\Longrightarrow}$.
To the contrary, suppose that \eqref{eex3} is not (iii$'$); that is,
$n<n_{\up s}$ or $k<k_{is_i}$. If $n<n_{\up s}$, then \eqref{eex3} is the composition of replacements of types (i) and (iii):
\[\mc L_n\oplus \mc L_{n_{\up s}}\oplus {\cal D}_k(\lambda_i )
               \re
 \mc L_n\oplus \mc L_{n_{\up s}+1}\oplus {\cal D}_{k-1}(\lambda_i )
               \re
\mc L_{n+1}\oplus \mc L_{n_{\up s}}\oplus {\cal D}_{k-1}(\lambda_i ).
 \]
If $k<k_{is_i}$, then
\begin{align*}
\mc L_{n}\oplus {\cal D}_k(\lambda_i)
\oplus {\cal D}_{k_{is_i}}(\lambda_i)
&\re
\mc L_{n+1} \oplus \mc D_{k}(\lambda_i )
\oplus {\cal D}_{k_{is_i}-1}(\lambda_i)
\\
&\re
\mc L_{n+1} \oplus \mc D_{k-1}(\lambda_i )
\oplus {\cal D}_{k_{is_i}}(\lambda_i).
\end{align*}
By Theorem I,
$\ll\mc B\rr$ is not an immediate successor of $\ll\mc A\rr$.

${\Longleftarrow}$.
Let $\mc B$ be obtained from $\mc A$ by replacement (iii$'$). Let
$\mc B$ can be also obtained from $\mc A$ by a sequence $\mc A=\mc A_1 \stackrel{\varphi_1}{\longmapsto}
\mc A_2
\stackrel{\varphi_2}{\longmapsto}
\cdots
\stackrel{\varphi_p}{\longmapsto}
\mc A_{p+1}=\mc B
$
of replacements of types (i)--(vi).

All replacements $\varphi _1,\dots,\varphi _p$ are not of
\begin{itemize}
  \item type (vi) since it decreases the number $\un s$;

  \item type (i) since it increases lexicographically $(n_1,n_2,\dots,n_{\un {s}})$;

  \item types (ii) and (iv) since they change the sequence $(m_1,m_2,\dots,m_{\up s})$;

\item type (v) with $\lambda =\lambda _l$ since it decreases lexicographically $(k_{i1},k_{i2},\dots,k_{is_i})$.
\end{itemize}
Therefore, all $\varphi _1,\dots,\varphi _p$ are of type (iii). Since each replacement (iii$'$) is not the composition of several replacements of type (iii), $p=1$, and so $\ll\mc B\rr$ immediately succeeds $\ll\mc A\rr$.
\medskip

\emph{Case 4:
{\rm (j)} is the replacement {\rm(iv)}.} The statement  \eqref{aam} is proved in this case by transposing the matrices in Case 3.
\medskip

\emph{Case 5:
{\rm (j)} is the replacement {\rm(v)}:}
\begin{equation}\label{eex5}
\mc D_{k_{ij}}(\lambda_i)\oplus\mc D_{k_{il}}(\lambda_i)\re\mc D_{k_{ij}-1}(\lambda_i)
\oplus \mc D_{k_{il}+1}(\lambda_i),\quad\text{in which }j<l .
\end{equation}\\[-30pt]

${\Longrightarrow}$.
To the contrary, suppose that \eqref{eex5} is not (v$'$); that is, $k_{ij}<k_{i,j+1}<k_{il}$.
Then
\begin{align*}
\mc D_{k_{ij}}(\lambda_i)\oplus\mc D_{k_{i,j+1}}(\lambda_i)\oplus\mc D_{k_{il}}(\lambda_i)&\re
\mc D_{k_{ij}-1}(\lambda_i)\oplus\mc D_{k_{i,j+1}+1}(\lambda_i)\oplus\mc D_{k_{il}}(\lambda_i)
\\
&\re \mc D_{k_{ij}-1}(\lambda_i)\oplus\mc D_{k_{i,j+1}}(\lambda_i)\oplus\mc D_{k_{il}+1}(\lambda_i).
\end{align*}
By Theorem I,
$\ll\mc B\rr$ is not an immediate successor of $\ll\mc A\rr$.

${\Longleftarrow}$.
Let $\mc B$ be obtained from $\mc A$ by replacement (v$'$), and let
$\mc B$ can be also obtained from $\mc A$ by a sequence $\mc A=\mc A_1 \stackrel{\varphi_1}{\longmapsto}
\mc A_2
\stackrel{\varphi_2}{\longmapsto}
\cdots
\stackrel{\varphi_p}{\longmapsto}
\mc A_{p+1}=\mc B
$
of replacements of types (i)--(vi).
All replacements $\varphi _1,\dots,\varphi _p$ are not of types (i)--(iv) and (vi) since they change $n_1,n_2,\dots,n_{\up s}$.

Therefore, all $\varphi _1,\dots,\varphi _p$ are of type (v). Since each replacement (v$'$) is not the composition of several replacements of type (v), $p=1$, and so $\ll\mc B\rr$  immediately succeeds $\ll\mc A\rr$.
\medskip

\emph{Case 6:
{\rm (j)} is the replacement {\rm(vi)}:}
\begin{equation}\label{eex6}
\mc L_{n_i}\oplus
\mc L^T_{m_j} \re
{\cal D}_{r_1} (\mu_1)\oplus\dots\oplus {\cal D}_{r_q}(\mu_q),
\end{equation}
in which $\mu_1,\dots,\mu_q\in{\mathbb C}\cup\infty$ are distinct and $r_1+\dots +r_q=n_i+m_j-1$.

${\Longrightarrow}$.
To the contrary, suppose that \eqref{eex6} is not (vi$'$).

If $n_i<n_{\up s}$, then
\begin{equation*}\label{cxl}
\begin{aligned}
\mc L_{n_i}\oplus \mc L_{n_{\up s}}\oplus
\mc L^T_{m_j}
 &\re
\mc L_{n_i}\oplus {\cal D}_{r_1+n_{\up s}-n_i} (\mu_1)\oplus\dots\oplus {\cal D}_{r_q}(\mu_q)
\\
                &\re
 \mc L_{n_{\up s}}\oplus{\cal D}_{r_1} (\mu_1)\oplus\dots\oplus {\cal D}_{r_q}(\mu_q),
\end{aligned}
\end{equation*}
and so $\ll\mc B\rr$ is not an immediate successor of $\ll\mc A\rr$. Hence $n_i=n_{\up s}$ and, analogously, $m_j=m_{\un s}$.

If some $\lambda _i\notin\{\mu _1,\dots,\mu _q\}$, then
\begin{align*}
\mc L_{n_{\up s}}\oplus
\mc L^T_{m_{\un s}}\oplus {\cal D}_{k_{i1}}(\lambda_i )
               &\re
               \mc L_{n_{\up s}+k_{i1}}\oplus
\mc L^T_{m_{\un s}}\\
                &\re
{\cal D}_{r_1} (\mu_1)\oplus\dots\oplus {\cal D}_{r_q}(\mu_q)
\oplus {\cal D}_{k_{i1}}(\lambda_i ),
\end{align*}
and so $\ll\mc B\rr$ is not an immediate successor of $\ll\mc A\rr$. Hence $q\ge t$ and we can rearrange  $\mu_1,\dots,\mu_q $ such that   $\mu_1=\lambda _1,\dots,\mu_t=\lambda _t$.

Let $r_i< k_{is_i}$ for some $i$; for definiteness, for $i=1$. Then $\mu_1=\lambda _1$,
\begin{align*}
\mc L_{n_{\up s}}\oplus
\mc L^T_{m_{\un s}}\oplus {\cal D}_{k_{1s_1}}(\mu_1)
               &\re
\mc L_{n_{\up s}+k_{1s_1}-r_1}\oplus
\mc L^T_{m_{\un s}}\oplus {\cal D}_{r_1}(\mu_1)
\\
                &\re
{\cal D}_{r_2} (\mu_2)\oplus\dots\oplus {\cal D}_{r_q}(\mu_q)
\oplus{\cal D}_{k_{1s_1}}(\mu_1)\oplus {\cal D}_{r_1}(\mu_1),
\end{align*}
and so $\ll\mc B\rr$ is not an immediate successor of $\ll\mc A\rr$. Hence, $r_1\ge  k_{1s_1}$, \dots, $r_t\ge  k_{ts_t}$.

${\Longleftarrow}$.
Let $\mc B$ be obtained from $\mc A$ by a replacement
\begin{equation}\label{tvt}
\varphi :\
\mc L_{n_{\up s}}\oplus
\mc L^T_{m_{\un s}} \re
{\mc D}_{r_1} (\mu_1)\oplus\dots\oplus {\mc D}_{r_q}(\mu_q),\qquad q\ge t
\end{equation}
of type (vi$'$); that is,
$\mu_1=\lambda _1$,\,\dots,\,$\mu_t=\lambda _t$, and $k_{1s_1}\le r_1$,\,\dots,\,$ k_{ts_t}\le r_t$.

Let $\mc B$ can be also obtained from $\mc A$ by a sequence
$\mc A=\mc A_1 \stackrel{\varphi_1}{\longmapsto}
\mc A_2
\stackrel{\varphi_2}{\longmapsto}
\cdots
\stackrel{\varphi_p}{\longmapsto}
\mc A_{p+1}=\mc B
$ of replacements of types (i)--(vi).
Exactly one replacement $\varphi _u:\mc A_u\to \mc A_{u+1}$ is of type (vi) since $\varphi$ decreases $\up s$ by one. The preceding replacements $\varphi_1,\dots,\varphi_{u-1}$ of types (i)--(v) do not change $\up s$ and $\un s$.
Let
  \begin{gather*}
\mc A':=\mc A_u=\bigoplus_{i=1}^{\un s}\mc L^T_{m'_i}\oplus\bigoplus_{i=1}^{\up s}
\mc L_{n'_i}\oplus
\bigoplus_{i=1}^{t'}
\Big(\mc D_{k'_{i1}}(\lambda_i)\oplus\dots\oplus
\mc D_{k'_{is'_i}}(\lambda_i)\Big),
                           \\
m'_1\le\cdots\le m'_{\underline s}, \ \
n'_1\le\cdots\le n'_{\overline s},\ \
k'_{i1}\le\cdots\le k'_{is'_{i}}\ (i=1,\dots,t'),\ \ t'\le t.
  \end{gather*}
We can suppose that $\varphi _u$ is not a product of replacements. Then  $\varphi _u$  is of type (iv$'$) due to part ``${\Longrightarrow}$''; that is,
\begin{equation*}\label{een}
\varphi _u:\
\mc L^T_{m'_{\un s}}\oplus
\mc L_{n'_{\up s}} \re
{\cal D}_{\rho_1} (\nu_1)\oplus\dots\oplus {\cal D}_{\rho_{q'}}(\nu_{q'}), \qquad q'\ge t',
\end{equation*}
in which $\nu_1=\lambda _1$,\,\dots,\,$\nu_{t'}=\lambda _{t'}$, and $k_{1s_1}\le \rho_1$,\,\dots,\,$ k_{t's_{t'}}\le \rho_{t'}$.

If $m'_{\un s}>m_{\un s}$, then $m_{\un s}$ has been increased by some $\varphi _l$ with $l<u$ of type (iv). However, this $\varphi _l$ decreases $\sum_{i,j} k_{ij}$, which cannot be restored because of the condition $k_{1s_1}\le r_1$,\,\dots,\,$ k_{ts_t}\le r_t$. Hence $m'_{\un s}\le m_{\un s}$ and, analogously, $n'_{\up s}\le n_{\up s}$.

If $m'_{\un s}< m_{\un s}$, then $\sum_{i,j}k'_{ij}+\sum_i\rho_i<
\sum_{i,j}k_{ij}+\sum_ir_i $ and this inequality cannot be transformed to the equality by replacements $\varphi_{u+1},\dots,\varphi_p$ of types (i)--(v).
Hence $m'_{\un s}=m_{\un s}$ and, analogously, $n'_{\up s}=n_{\up s}$.

If $\rho _1<r_1$, then
$$k'_{11}+\dots+k'_{1s'_1}+\rho_1
<k_{11}+\dots+k_{1s_1}+r_1,$$
and this inequality cannot be transformed to the equality by replacements $\varphi_{u+1},\dots,\varphi_p$ of types (i)--(v). Hence $\rho_1\ge r_1$ and, analogously, $\rho_i\ge r_i$ for all $i$. Using $m'_{\un s}=m_{\un s}$ and $n'_{\up s}=n_{\up s}$, we find that $t'=t$ and $\rho_i= r_i$ for all $i$. Therefore, $\varphi _u$ is the replacement $\varphi $ from \eqref{tvt}. It is easy to check that $u=p=1$ and $\varphi _1=\varphi$.


\begin{thebibliography}{99}

\bibitem{ala}
A. Alazemi, M. An\dj eli\'{c}, C.M. da Fonseca, V.V. Sergeichuk, Lipschitz property for systems of linear mappings and bilinear forms, Linear Algebra Appl. 573 (2019) 26--36.

\bibitem{arn} V.I. Arnold, On matrices
    depending on parameters, {Russian Math. Surveys} 26 (2)
    (1971) 29--43.

\bibitem{arn2} V.I. Arnold, Lectures on
    bifurcations in versal
    families, {Russian Math.
    Surveys}
    27 (5) (1972) 54--123.

\bibitem{arn3} V.I. Arnold,
    {Geometrical Methods in the Theory
    of Ordinary Differential
    Equations}, Springer-Verlag, 1988.

\bibitem{ben}
J. Bender, K. Bongartz, Minimal singularities in orbit closures of matrix pencils, Linear Algebra Appl. 365 (2003) 13--24.

\bibitem{bol}
D.L. Boley, The algebraic structure of pencils and block Toeplitz matrices, Linear Algebra Appl.
279 (1998) 255--279.

\bibitem{bon_1}
K. Bongartz,
Minimal singularities for representations of Dynkin quivers,
Comment. Math. Helv. 69 (1994) 575--611.

\bibitem{bon_2}
K. Bongartz, Degenerations for representations of tame quivers, Ann. Scient. Ec. Norm. Sup., 4e serie 28 (1995) 647--688.

\bibitem{bon}
 K. Bongartz,  On degenerations and extensions of finite dimensional modules,  Adv.
Math. 121 (1996) 245--287.

\bibitem{bon_3}
K. Bongartz,
Some geometric aspects of representation theory, in: Algebras and modules, I (Trondheim, 1996), 1--27,
CMS Conf. Proc., 23, Amer. Math. Soc., Providence, RI, 1998.

\bibitem{bov}
V.A. Bovdi, M.A. Salim, V.V. Sergeichuk,
Neighborhood radius estimation for Arnold's miniversal deformations of complex and $p$-adic
matrices, Linear Algebra Appl. 512 (2017) 97--112.

\bibitem{de}
F. De Ter\'{a}n, F.M. Dopico, A note on generic Kronecker orbits of matrix pencils with fixed rank,
SIAM J. Matrix Anal. Appl. 30 (2008) 491--496.

\bibitem{den}
H.  Den  Boer,  G.Ph.A. Thijsse,  Semi-stability  of sums of partial  multiplicities
under additive  perturbation,  Integral Equations  Operator Theory 3 (1980) 23--42.


\bibitem{dmy1}
 A. Dmytryshyn, Miniversal deformations of pairs of skew-symmetric matrices under congruence,
Linear Algebra Appl. 506 (2016) 506--534.

\bibitem{dmy2}
 A. Dmytryshyn, Miniversal deformations of pairs of symmetric matrices under congruence,
Linear Algebra Appl. 568 (2019) 84--105.

\bibitem{dmy}
A. Dmytryshyn, F.M. Dopico, Generic skew-symmetric matrix polynomials with fixed rank and fixed odd grade, Linear Algebra Appl. 536 (2018) 1--18.


\bibitem{f_s}  A.R.
    Dmytryshyn, V.
    Futorny, V.V.
    Sergeichuk, Miniversal deformations
    of matrices of bilinear forms,
    Linear Algebra Appl. 436 (2012)
    2670--2700.

\bibitem{def-sesq} A.R.
    Dmytryshyn, V.
    Futorny, V.V.
    Sergeichuk,
    Miniversal deformations of matrices
    under *congruence and reducing
    transformations, Linear Algebra
    Appl. 446 (2014)
    388--420.

\bibitem{d-f-k} A.
    Dmytryshyn, V.
    Futorny, B. K\r{a}gstr\"{o}m,
    L. Klimenko,
    V.V. Sergeichuk, Change of
the congruence canonical form of 2-by-2
and 3-by-3 matrices under perturbations
and bundles of matrices under
congruence, Linear Algebra
Appl. 469 (2015) 305--334.

\bibitem{kag} A. Edelman, E. Elmroth,
    B. K\r{a}gstr\"{o}m, A geometric
    approach to perturbation theory of
    matrices and matrix pencils. Part
    I: Versal deformations, SIAM
    J. Matrix Anal. Appl. 18
    (1997) 653--692.

\bibitem{kag2} A. Edelman, E. Elmroth,
    B. K\r{a}gstr\"{o}m, A geometric
    approach to perturbation theory of
    matrices and matrix pencils. Part
    II: A stratification-enhanced staircase algorithm, SIAM J. Matrix Anal. Appl. 20
    (1999)  667--699.

\bibitem{f-j-k}
E. Elmroth,  P. Johansson, B. K\r{a}gstr\"{o}m,
Computation and presentation of graphs displaying closure hierarchies of Jordan and Kronecker structures,
Numer. Linear Algebra Appl. 8 (2001) 381--399.

\bibitem{f-k}
E. Elmroth, B. K\r{a}gstr\"{o}m,  The set of 2-by-3 matrix pencils --- Kronecker structures and
their transitions under perturbations,  SIAM J. Matrix Anal. Appl. 17 (1996) 1--34.

\bibitem{f-k-s}
V. Futorny, L. Klimenko,
V.V. Sergeichuk, Change of the
*congruence canonical form of 2-by-2
matrices under perturbations, Electr.
J. Linear Algebra 27 (2014) 146--154.

\bibitem{gar_mai} M.I.
    Garc\'{\i}a-Planas,
    A.A. Mailybaev, Reduction to versal
    deformations of matrix pencils and
    matrix pairs with application to
    control theory, SIAM J. Matrix Anal. Appl. 24
    (2003) 943--962.

\bibitem{gar_ser} M.I. Garc\'{\i}a-Planas,
    V.V. Sergeichuk, Simplest
    miniversal deformations of
    matrices, matrix pencils, and
    contragredient matrix pencils,
    {Linear Algebra Appl.} 302--303
    (1999) 45--61.


\bibitem{hin}
D.  Hinrichsen,     J.  O'Halloran,      Orbit   closures    of  singular   matrix    pencils,   J.  Pure   Appl.   Algebra
81  (1992)   117--137.

\bibitem{joh}
S. Johansson, Reviewing the Closure Hierarchy of Orbits and Bundles of System Pencils and Their Canonical Forms, Technical report, Ume{\aa} University, Department of Computing Science, 2009, UMINF-09.02.


\bibitem{K-J}
B. K\r{a}gstr\"{o}m, S. Johansson, P. Johansson, StratiGraph tool: matrix stratifications in control applications, in: L.T. Biegler et al. (Eds),  Control and optimization with differential-algebraic constraints, 79--103, Adv. Des. Control, 23, SIAM, Philadelphia, PA, 2012.

\bibitem{K-S}
L. Klimenko, V.V. Sergeichuk, Block triangular miniversal deformations of matrices and matrix pencils, in:  V.  Olshevsky,  E.  Tyrtyshnikov  (Eds),  Matrix  Methods:  Theory,  Algorithms  and Applications, World Sci. Publ., Hackensack, NJ, 2010, pp. 69--84.

\bibitem{k-s}
L. Klimenko, V.V. Sergeichuk, An informal introduction to perturbations of
matrices determined up to similarity or
congruence, S\~ao Paulo  J. Math.
Sci. 8 (2014) 1--22.


\bibitem{mai1999} A.A.
    Mailybaev,
    Reduction of matrix families to normal forms and its application to stability theory, Fundam. Prikl. Mat. 5 (1999) 1111--1133 (in Russian).

\bibitem{mai2000} A.A.
    Mailybaev,
    Transformation of families of
    matrices to normal forms and its
    application to stability theory,
    SIAM J. Matrix Anal. Appl. 21
    (1999) 396--417.

\bibitem{mai2001} A.A.
    Mailybaev,
    Transformation to versal
    deformations of matrices, Linear
    Algebra Appl. 337
    (2001) 87--108.

\bibitem{mar}
A.S. Markus,  E.\`{E}. Parilis, The change of the Jordan structure of a matrix under small perturbations, Mat. Issled, No. 54 (1980) 98--109 (in Russian), English thranslation:   Linear Algebra Appl. 54 (1983) 139--152.

\bibitem{omear}
 K. O'Meara, J. Clark, C. Vinsonhaler, Advanced  Topics  in Linear Algebra: Weaving Matrix Problems Through the Weyr Form, Oxford University Press, 2011.

\bibitem{pok}
A. Pokrzywa, On  perturbations and the equivalence orbit of a matrix pencil, Linear Algebra Appl. 82 (1986) 99--121.

\bibitem{rod}
L. Rodman, Remarks on Lipschitz properties of matrix groups actions,
Linear Algebra Appl. 434 (2011) 1513--1524.

\bibitem{ser} V.V. Sergeichuk,
    Canonical matrices for linear
    matrix problems, Linear Algebra
    Appl. 317 (2000) 53--102.

\bibitem{str}
StratiGraph and MCS Toolbox, Software Tools,
Department of Computing Science, Ume{\aa} University, Sweden, [Website],\\
{\small
www.cs.umu.se/english/research/groups/matrix-computations/stratigraph}


\end{thebibliography}
\end{document}